\documentclass[leqno]{article} 

\usepackage[english]{babel}
\usepackage[T1]{fontenc}

\usepackage{amsmath}
\usepackage{amssymb}
\usepackage{amsfonts}
\usepackage{enumerate}
\usepackage{vmargin}
\usepackage[all]{xy}
\usepackage{mathrsfs}
\usepackage{mathtools}
\usepackage{lmodern}
\usepackage[pagebackref=true]{hyperref}

\setmarginsrb{3cm}{3cm}{3.5cm}{3cm}{0cm}{0cm}{1.5cm}{3cm}

\newcommand{\C}{\ensuremath{\mathbb{C}}}

\newcommand{\E}{\ensuremath{\mathbb{E}}}

\newcommand{\N}{\ensuremath{\mathbb{N}}}
\newcommand{\R}{\ensuremath{\mathbb{R}}}

\newcommand{\W}{\ensuremath{\mathcal{W}}}
\newcommand{\Z}{\ensuremath{\mathbb{Z}}}

\renewcommand{\leq}{\ensuremath{\leqslant}}
\renewcommand{\geq}{\ensuremath{\geqslant}}
\newcommand{\qed}{\hfill \vrule height6pt  width6pt depth0pt}

\newcommand{\supp}{\mathrm{supp}}

\newcommand{\dist}{\mathrm{dist}}
\newcommand{\Id}{\mathrm{Id}}
\newcommand{\Prob}{\mathrm{Prob}}

\newcommand{\HI}{H^\infty}
\newcommand{\Hor}{\mathcal{H}}
\newtheorem{thm}{Theorem}[section]
\newtheorem{defi}[thm]{Definition}

\newtheorem{prop}[thm]{Proposition}

\newtheorem{cor}[thm]{Corollary}
\newtheorem{lemma}[thm]{Lemma}

\newtheorem{remark}[thm]{Remark}

\newtheorem{ass}[thm]{Assumption}

\newenvironment{proof}[1][]{\noindent {\it Proof #1} : }{\hbox{~}\qed
\smallskip
}

\numberwithin{equation}{section}

\begin{document}
\selectlanguage{english}
\title{\bfseries{H\"ormander functional calculus on UMD lattice valued $L^p$ spaces under generalised Gaussian estimates}}
\date{January 2018}
\author{\bfseries{Luc Deleaval, Mikko Kemppainen and Christoph Kriegler}}

\maketitle

\begin{abstract}
We consider self-adjoint semigroups $T_t = \exp(-tA)$ acting on $L^2(\Omega)$ and satisfying (generalised) Gaussian estimates, where $\Omega$ is a metric measure space of homogeneous type of dimension $d$.
The aim of the article is to show that $A \otimes \Id_Y$ admits a H\"ormander type $\Hor^\beta_2$ functional calculus on $L^p(\Omega;Y)$ where $Y$ is a UMD lattice, thus extending the well-known H\"ormander calculus of $A$ on $L^p(\Omega)$.
We show that if $T_t$ is lattice positive (or merely admits an $\HI$ calculus on $L^p(\Omega;Y)$) then this is indeed the case.
Here the derivation exponent has to satisfy $\beta > \alpha \cdot d + \frac12$, where $\alpha \in (0,1)$ depends on $p$, and on convexity and concavity exponents of $Y$.
A part of the proof is the new result that the Hardy-Littlewood maximal operator is bounded on $L^p(\Omega;Y)$.
Moreover, our spectral multipliers satisfy square function estimates in $L^p(\Omega;Y)$.
In a variant, we show that if $e^{itA}$ satisfies a dispersive $L^1(\Omega) \to L^\infty(\Omega)$ estimate, then $\beta > \frac{d+1}{2}$ above is admissible independent of convexity and concavity of $Y$.
Finally, we illustrate these results in a variety of examples.
\end{abstract}


\makeatletter
 \renewcommand{\@makefntext}[1]{#1}
 \makeatother
 \footnotetext{
 {\it Mathematics subject classification:}
 42A45, 42B25, 47A60, 47A80.
\\
{\it Key words}: Spectral multiplier theorems, UMD valued $L^p$ spaces, Gaussian estimates.}

 \tableofcontents

\section{Introduction}

Let $f$ be a bounded function on $(0,\infty)$ and $u(f)$ the operator on $L^p(\R^d)$ defined by $[f(-\Delta)g]\hat{\phantom{i}} = [u(f)g]\hat{\phantom{i}}=f(\|\xi\|^2)\hat{g}(\xi).$
H\"ormander's theorem on Fourier multipliers \cite[Theorem 2.5]{Ho60} asserts that $u(f) : L^p(\R^d) \to L^p(\R^d)$ is bounded for any $p \in (1,\infty)$ provided that for some integer $\beta$ strictly larger than $\frac{d}{2},$
\begin{equation}
\label{equ-intro-classical-Hormander}
\|f\|_{\Hor^\beta_2}^2 := \max_{k=0,1,\ldots,\beta} \sup_{R > 0} \frac{1}{R}\int_{R}^{2R} \Bigl| t^k \frac{d^k}{dt^k} f(t) \Bigr|^2 \,dt < \infty.
\end{equation}
This theorem has many refinements and generalisations to various similar contexts.
Namely, one can generalise to non-integer $\beta$ in \eqref{equ-intro-classical-Hormander} to get larger (for smaller $\beta$) admissible classes $\Hor^\beta_2 = \{ f : (0,\infty) \to \C \text{ bounded and continuous} :\: \|f\|_{\Hor^\beta_2} < \infty \}$ of multiplier functions $f$ (see Subsection \ref{subsec-R-bdd}).
Moreover, it has been a deeply studied question over the last years to know to what extent one can replace the ordinary Laplacian subjacent to H\"ormander's theorem by other operators $A$ acting on some $L^p(\Omega)$ space.
A theorem of H\"ormander type holds true for many elliptic differential operators $A,$ including sub-Laplacians on Lie groups of polynomial growth, Schr\"odinger operators and elliptic operators on Riemannian manifolds, see \cite{Alex,Christ,Duong,DuOS}.
More recently, spectral multipliers have been studied for operators acting on $L^p(\Omega)$ only for a strict subset of $(1,\infty)$ of exponents \cite{Bl,CDY,CO,COSY,KuUhl,KU2,SYY} and for abstract operators acting on Banach spaces \cite{KrW3}.
A spectral multiplier theorem means then that the linear and multiplicative mapping
\begin{equation}
\label{equ-intro-homomorphism}
\Hor^\beta_2 \to B(X), \: f \mapsto f(A),
\end{equation}
is bounded, where typically $X = L^p(\Omega).$

The main topic of the present article is to determine in which cases \eqref{equ-intro-homomorphism} holds with $X = L^p(\Omega;Y)$, i.e. when does the tensor extension $f(A) \otimes \Id_Y$ of $f(A) : L^p(\Omega) \to L^p(\Omega)$ extend to a bounded operator on the Bochner space $L^p(\Omega;Y)$.
It is well-known that if $Y$ is a Hilbert space, or if $f(A)$ is lattice positive, or if $f(A)$ is both bounded on $L^\infty(\Omega)$ and on $L^1(\Omega)$, then this tensor extension is possible, but in general, this is a difficult task, e.g. for a multiplier $f(A)$ with singular integral kernel having a cancellation effect.
As a motivation for this question, take the following abstract hyperbolic PDE
\begin{align*}
\partial_t^2 u(x,y,t) & = - A_x u(x,y,t) & \quad (x \in \Omega, \: y \in \Omega' \: t > 0) \\
u(x,y,0) & = f(x,y) & \quad (x \in \Omega, \: y \in \Omega') ,
\end{align*}
which is solved formally by $u(x,y,t) = \exp(it\sqrt{A})(f)(x,y)$.
Noting that \[f(\lambda) = (1 + \lambda)^{-\delta} \exp(it \sqrt{\lambda})\] belongs to the class $\Hor^\beta_2$ for $\delta \geq \frac{\beta}{2}$ \cite[Lemma 3.9]{KrW3}, \cite[Prop 4.8 (4)]{KrPhD} yields that $\|u(t)\|_X \leq C (1+|t|)^{2 \delta} \|(1+A)^{\delta}f\|_X$ provided that \eqref{equ-intro-homomorphism} holds.
One thus obtains a norm estimate of the solution $u$ in terms of fractional domain space norms $D((1+A)^\delta) \subset X$ of the initial value $f$.
As an example, we can take $X = L^p(\Omega;L^s(\Omega'))$.
We refer to \cite[Sections 5 and 6]{HiPr} and \cite[Section 4]{DeKr1} for further applications of the functional calculus to equations on such $X$.

Even in the case that $A = - \Delta$, which can be considered as our basic example and starting point for further considerations, one cannot take any Banach space $Y$ in $X = L^p(\Omega;Y)$ of \eqref{equ-intro-homomorphism}, but is restricted to take a UMD space $Y$ \cite[10.3 Remark]{KW04}, for a definition of the UMD property see Subsection \ref{subsec-UMD}.
In establishing a functional calculus on $X = L^p(\Omega)$, square function estimates such as 
\begin{equation}
\label{equ-square-function-intro}
\Bigl\| \Bigl( \sum_k |T_{t_k} f_k |^2 \Bigr)^{\frac12} \Bigr\|_X \leq C \Bigl\| \Bigl( \sum_k |f_k|^2 \Bigr)^{\frac12} \Bigr\|_X 
\end{equation}
are known to play an important role, where $T_t$ is a spectral multiplier of $A$, typically the semigroup generated by $A$.
We will prove such square function estimates also for $X = L^p(\Omega;Y)$.
In order to do so, we will need a maximal estimate, which in the simplest form states as $|T_tf| \leq c M_{HL}(f)$ for all $t \geq 0$. 
A natural framework for us will be that $Y = Y(\Omega')$ is a \textit{UMD lattice} over some measure space $\Omega'$ and $\Omega$ a metric measure space, in fact, a space of homogeneous type, see Subsection \ref{subsec-GE}.
Then $M_{HL}$ stands for the Hardy-Littlewood maximal operator, which is 
\[ M_{HL} (f)(x,\omega') = \sup_{r > 0} \frac{1}{V(x,r)} \int_{B(x,r)} |f(y,\omega')| \, d\mu(y) \quad ( x \in \Omega,\: \omega' \in \Omega') \]
where $B(x,r)$ stands for the closed ball centered in $x$ of radius $r$ and $V(x,r)$ stands for the volume of that ball.
Our first main result reads then as:

\begin{thm}
\label{thm-MHL-intro}
If $\Omega$ is a space of homogeneous type, then $M_{HL}$ is bounded on $L^p(\Omega;Y)$ for any $p \in (1,\infty)$ and for every UMD lattice $Y$.
\end{thm}

By an abstract machinery from \cite{KrW3}, square functions as in \eqref{equ-square-function-intro} for $T_t$ the semigroup for complex times $t \in \C_+$ plus an a priori $\HI$ calculus of $A$ on $L^p(\Omega;Y)$ are sufficient for the H\"ormander calculus \eqref{equ-intro-homomorphism}.
Our task thus becomes to verify \eqref{equ-square-function-intro} and the $\HI$ calculus of $A$, which means that \eqref{equ-intro-homomorphism} holds with the class $\Hor^\beta_2$ replaced by the smaller class $\HI(\Sigma_\omega)$ consisting of bounded and analytic functions, see Subsection \ref{subsec-R-bdd} for a precise definition.
For the $\HI$ calculus, there are several strategies of extrapolation, or lattice positivity of the semigroup, or $L^\infty(\Omega)$ and $L^1(\Omega)$ contractivity of the semigroup, see Theorem \ref{thm-HI-extrapolation}, Corollary \ref{cor-HI-positive} or Proposition \ref{prop-HI-diffusion} respectively.
When it comes to the square function estimate \eqref{equ-square-function-intro}, we have chosen as a starting point that $\Omega$ is a space of homogeneous type, $T_t$ is self-adjoint on $L^2(\Omega)$, has a representation 
\[T_tf(x) = \int_\Omega p_t(x,y) f(y) \, d\mu(y)\]
and its integral kernel $p_t(x,y)$ satisfies Gaussian estimates, that is, for some $m \geq 2,\:C ,c > 0$
\begin{equation}
\label{equ-GE-intro}
|p_t(x,y)| \leq C \frac{1}{V(x,r_t)} \exp\Biggl(-c \biggl(\frac{\dist(x,y)}{r_t}\biggr)^{\frac{m}{m-1}} \Biggr) \quad (x,y \in \Omega,\: t > 0),
\end{equation}
where $r_t = t^{\frac1m}$.
Such estimates are by now a well-established property for semigroups generated by differential operators (see Section \ref{sec-examples}).
These are not the most general assumptions for our purposes and the more general generalised Gaussian estimates, see  \eqref{equ-GGE} below, for semigroups acting only on $L^p(\Omega)$ for $p$ belonging to a subinterval $(p_0,p_0') \subseteq (1,\infty)$ fit equally well for our methods.
See Subsection \ref{subsec-GGE} for examples.
We have chosen a presentation of our method in Section \ref{sec-main} for generalised Gaussian estimates and deduce the classical Gaussian estimates case as corollaries.
Then our main result reads as follows.

\begin{thm}
\label{thm-Hoermander-intro}
Let $(\Omega,\dist,\mu)$ be a space of homogeneous type with a dimension $d$.
Let $A$ be a self-adjoint operator on $L^2(\Omega)$ generating the semigroup $(T_t)_{t \geq 0}$.
Assume that $(T_t)_{t \geq 0}$ satisfies Gaussian estimates \eqref{equ-GE-intro} with parameter $m \geq 2$.
Let $Y$ be a UMD lattice.
Finally, assume that $T_t$ is lattice positive, i.e $T_tf \geq 0$ for all $f \geq 0$ and all $t \geq 0$, or merely that $A$ has a bounded $\HI(\Sigma_\omega)$ calculus on $L^p(\Omega;Y)$ for some fixed $p \in (1,\infty)$ and $\omega \in (0,\pi)$.
Then $A$ has a H\"ormander $\Hor^\beta_2$ calculus on $L^p(\Omega;Y)$ with
\[ \beta > \alpha \cdot d + \frac12.\]
\end{thm}

Here, $\alpha \in (0,1)$ is a parameter depending on $p$ and $Y$, and will be close to the best value $0$ if $p$ is close to $2$ and $Y$ is close to a Hilbert space, e.g. $Y = L^s(\Omega')$ with $s$ close to $2$.
More generally, $\alpha$ is a function of $p$ and of the convexity and concavity index of the lattice $Y$ in the sense of \cite{LTz} (see \eqref{equ-defi-alpha}, and Subsection \ref{subsec-UMD} for the definition of these notions).
Theorem \ref{thm-Hoermander-intro} is proved in Corollaries \ref{cor-Hoermander} and \ref{cor-HI-positive}.
We refer to Theorem \ref{thm-Hoermander} for the version with $(p_0,m)$ generalised Gaussian estimates, which needs convexity and concavity exponents of $Y$ compatible with $p_0$.
In the scalar case $Y = \C$ (or Hilbert space case), the derivation index of Theorem \ref{thm-Hoermander-intro} becomes $\beta > |\frac1p - \frac12| d + \frac12$, compared to the better $\beta > \max(|\frac1p - \frac12| d,\frac12)$ known in many cases, see e.g. \cite[Theorem 4.1]{COSY} and Remark \ref{rem-comparison}.
This price of higher differentiation order is justified not only by the fact that $Y$ can be a lattice for us, but also by the following strengthening of Theorem \ref{thm-Hoermander-intro} to square function estimates that we obtain:

\begin{thm}
\label{thm-Hoermander-square-intro}
Assume that the hypotheses of Theorem \ref{thm-Hoermander-intro} above hold and let $\beta > \alpha \cdot d + \frac12$ be as in the conclusion.
Then there is $C < \infty$ such that
\begin{equation}
\label{equ-R-bdd-Hor-intro}
\biggl\| \biggl( \sum_k |m_k(A) f_k|^2 \biggr)^{\frac12} \biggr\|_{L^p(\Omega;Y)} \leq C \sup_{k} \|m_k\|_{\Hor^\beta_2} \biggl\| \biggl( \sum_k |f_k|^2 \biggr)^{\frac12} \biggr\|_{L^p(\Omega;Y)} .
\end{equation}
\end{thm}

This can be rephrased as $\{ m(A) : \: \|m\|_{\Hor^\beta_2} \leq 1 \}$ is $R$-bounded in $L^p(\Omega;Y)$ (see Subsection \ref{subsec-R-bdd} for the definition).
Concerning our method, there might also be weaker Poisson estimates as e.g. in \cite{DuRo,Kr1} sufficient, which have a polynomial decay at $\dist(x,y) \to \infty$ in place of the exponential decay in the Gaussian estimates.
But since they are rarely known for complex times that we would need (see \cite[Section 4]{Kr1} for examples), we have chosen not to pursue this case in our presentation.

In case that one does not know the convexity and concavity exponents of $Y$, there is an alternate approach to get to \eqref{equ-square-function-intro}, which has as additional assumption a dispersive estimate \eqref{equ-intro-dispersive}.
Namely we shall show the following.

\begin{thm}
\label{thm-dispersive-intro}
Let $(\Omega,\dist,\mu)$ be a space of homogeneous type of dimension $d$, let $Y$ be any UMD lattice and let $1 < p < \infty$.
Assume that $A$ generates the self-adjoint semigroup $(T_t)_{t \geq 0}$ on $L^2(\Omega)$ satisfying the Gaussian estimate \eqref{equ-GE-intro} with $m = 2$ and that $(T_t)_{t \geq 0}$ is lattice positive.
Assume moreover that there is a polynomial volume growth $V(x,r) \leq C |r|^d$ ($x \in \Omega,\: r > 0$) and that $A$ satisfies the dispersive estimate
\begin{equation}
\label{equ-intro-dispersive}
 \| \exp(itA) \|_{L^1(\Omega) \to L^\infty(\Omega)} \leq C |t|^{-\frac{d}{2}} \quad (t \in \R \backslash \{ 0 \} ).
\end{equation}
Then $A$ has a $\Hor^\beta_2$ calculus on $L^p(\Omega;Y)$ for any exponent $\beta > \frac{d}{2} + \frac12$.
Moreover, \eqref{equ-R-bdd-Hor-intro} holds.
\end{thm}

For a proof of this theorem, we refer to Corollaries \ref{cor-HI-positive} and \ref{cor-GE-Hoermander}, where a slightly more general formulation than Theorem \ref{thm-dispersive-intro} is used.
Let us remark that the above Theorems \ref{thm-Hoermander-intro}, \ref{thm-Hoermander-square-intro} and \ref{thm-dispersive-intro} are valid if $Y$ is a UMD space \textit{isomorphic} to a (UMD) lattice (having convexity and concavity if applicable).
This is e.g. the case if $Y$ is a UMD space with an unconditional basis \cite[p.~19]{LTz1}.
Then for Theorem \ref{thm-Hoermander-square-intro}, one has to rewrite the square functions by Rademacher sums as in \eqref{equ-Rademacher-square} below.

We conclude this introduction with an overview of the sections.
In Section \ref{sec-prelims}, we introduce the necessary background of the mathematical objects we study.
In particular, in Subsection \ref{subsec-R-bdd}, we introduce square function estimates as used in \eqref{equ-square-function-intro} in Banach spaces, define the H\"ormander function space and give some simple properties.
Moreover, we indicate how one can define the H\"ormander calculus \eqref{equ-intro-homomorphism} for a semigroup generator, without using the self-adjoint calculus as a starting point, which is missing for Bochner spaces $L^p(\Omega;Y)$ unless $Y$ itself is a Hilbert space.
Moreover, we give in Theorem \ref{thm-KrW3} the above mentioned abstract criterion for the H\"ormander calculus from \cite{KrW3}.
In Subsection \ref{subsec-UMD}, we discuss the framework for the UMD lattice $Y$ and the needed properties, e.g. related to convexity and concavity.
Next, we give in Subsection \ref{subsec-tensor} the indications how a tensor amplificated operator $T \otimes \Id_Y : L^p(\Omega) \otimes Y \to L^p(\Omega) \otimes Y$ extends properly to $L^p(\Omega;Y)$.
In Subsection \ref{subsec-GE}, we recall the class of metric measure spaces $\Omega$ that we use throughout as well as the above mentioned (generalised) Gaussian estimates \eqref{equ-GE-intro}.
Section \ref{sec-MHL} is entirely devoted to prove Theorem \ref{thm-MHL-intro}.
Then Section \ref{sec-main} contains the main material to prove the above Theorems \ref{thm-Hoermander-intro} and \ref{thm-Hoermander-square-intro}, and we also give several sufficient criteria when $A$ admits the a priori needed $\HI$ calculus on $L^p(\Omega;Y)$.
For the reader's convenience, we spell out the H\"ormander theorem in the important particular cases of Gaussian estimates and of the case $Y = L^s(\Omega')$ with $\Omega'$ a further measure space (Corollary \ref{cor-Hoermander}). 
We compare our results with those from other literature (Remark \ref{rem-comparison}).
Moreover, we indicate some consequences of Theorem \ref{thm-Hoermander-intro} on Bochner-Riesz means and Paley-Littlewood decompositions in Bochner spaces.
The proof of Theorem \ref{thm-Hoermander-intro} is based on extrapolation of (generalised) off-diagonal estimates from real to complex time of the self-adjoint semigroup generated by $A$ stemming from Blunck's and Kunstmann's work, combined with the Hardy-Littlewood maximal operator theorem \ref{thm-MHL-intro}, and uses the non-trivial convexity and concavity of $Y$.
A different approach without the extrapolation procedure and hence without referring to convexity and concavity is given in Subsection \ref{subsec-simpler-approach}.
Here, we present three cases in which complex time estimates of the semigroup are known: first, the pure Laplacian case, which is one of the rare cases where the complex time integral kernel is explicitly known, second, the extrapolation of Gaussian estimates from \cite{CaCoOu} and third, the case of Gaussian estimates combined with a dispersive estimate, see Corollary \ref{cor-GE-Hoermander}.
In the first and the third case, a H\"ormander calculus theorem is derived with an exponent $\frac{d+1}{2}$, which is worse than the one from Theorem \ref{thm-Hoermander-intro} in case that $L^p(\Omega;Y)$ is close to being Hilbert, but is better in case that no particular information on convexity and concavity of $Y$ is given.
Finally in Section \ref{sec-examples}, we illustrate a variety of cases when our Theorems \ref{thm-Hoermander-intro}, \ref{thm-Hoermander-square-intro} and \ref{thm-dispersive-intro} apply.

\section{Preliminaries}
\label{sec-prelims}
In this section, we define and recall the central notions of the article and we prove several lemmas and results which will be relevant for the sequel. We first begin with preliminaries on $R$-boundedness, $\HI$ functional calculus and H\"ormander functional calculus.

\subsection{$R$-boundedness, $\HI$ functional calculus and H\"ormander functional calculus}
\label{subsec-R-bdd}

\begin{defi}
Let $X$ be a Banach space and $\tau \subset B(X).$
Then $\tau$ is called $R$-bounded if there is some $C < \infty$ such that for any $n \in \N,$ any $x_1,\ldots,x_n \in X$ and any $T_1,\ldots,T_n \in \tau,$ we have
\[ \E \biggl\|\sum_{k = 1}^n \epsilon_k T_k x_k \biggr\|_X \leq C \E \biggl\|\sum_{k = 1}^n \epsilon_k x_k \biggr\|_X,\]
where the $\epsilon_k$ are i.i.d. Rademacher variables on some probability space, that is, $\Prob(\epsilon_k = \pm 1) = \frac12.$
The least admissible constant $C$ is called $R$-bound of $\tau$ and is denoted by $R(\tau).$
\end{defi}

Note that trivially, we always have $R(\{ T \}) = \|T\|$ for any $T \in B(X).$
Although the notion of $R$-boundedness is stated in the literature usually only for families of linear operators, it makes literally perfectly sense for non-linear mappings $X \to X$, and we shall use it later for sublinear operators on Banach lattices.

\begin{defi}
\label{defi-lower-R-bounded}
Let $X$ be a Banach space.
We call a family $\tau$ of (in general non-linear) mappings $X \to X$ lower $R$-bounded if there exists a $C < \infty$ such that for any $x_1,x_2,\ldots, x_n \in X$ and $T_1,T_2,\ldots, T_n \in \tau,$ we have
\[ \E \biggl\|  \sum_{k = 1}^n \epsilon_k x_k \biggr\|_{X} \leq C \biggl\| \sum_{k = 1}^n \epsilon_k T_k x_k \biggr\|_{X}. \]
\end{defi}

We next recall the necessary background on functional calculus that we will treat in this article.
Let $-A$ be a generator of an analytic semigroup $(T_z)_{z \in \Sigma_\delta}$ on some Banach space $X,$ that is, $\delta \in (0,\frac{\pi}{2}],$ $\Sigma_\delta = \{ z \in \C \backslash \{ 0 \} :\: | \arg z | < \delta \},$ the mapping $z \mapsto T_z$ from $\Sigma_\delta$ to $B(X)$ is analytic, $T_{z+w} = T_z T_w$ for any $z,w \in \Sigma_\delta,$ and $\lim_{z \in \Sigma_{\delta'},\:|z| \to 0 } T_zx = x$ for any strict subsector $\Sigma_{\delta'}.$
We assume that $(T_z)_{z \in \Sigma_\delta}$ is a bounded analytic semigroup, which means $\sup_{z \in \Sigma_{\delta'}} \|T_z\| < \infty$ for any $\delta' < \delta.$

It is well-known \cite[Theorem 4.6, p. 101]{EN} that this is equivalent to $A$ being pseudo-$\omega$-sectorial for $\omega = \frac{\pi}{2} - \delta,$ that is,
\begin{enumerate}
\item $A$ is closed and densely defined on $X;$
\item The spectrum $\sigma(A)$ is contained in $\overline{\Sigma_\omega}$ (in $[0,\infty)$ if $\omega = 0$);
\item For any $\omega' > \omega,$ we have $\sup_{\lambda \in \C \backslash \overline{\Sigma_{\omega'}}} \| \lambda (\lambda - A)^{-1} \| < \infty.$
\end{enumerate}
We say that $A$ is $\omega$-sectorial if it is pseudo-$\omega$-sectorial and has moreover dense range.
If $A$ is pseudo-$\omega$-sectorial and does not have dense range, but $X$ is reflexive, which will always be the case in this article, then we may take the injective part $A_0$ of $A$ on $\overline{R(A)} \subseteq X$ \cite[Proposition 15.2]{KW04}, which then does have dense range and is $\omega$-sectorial.
Here, $R(A)$ stands for the range of $A.$
Then $-A$ generates an analytic semigroup on $X$ if and only if so does $-A_0$ on $\overline{R(A)}.$
This parallel will continue this section, i.e. the functional calculus for $A_0$ can be extended to $A$ in an obvious way, see \cite[Illustration 4.87]{KrPhD}.

For $\theta \in (0,\pi),$ let 
\[ \HI(\Sigma_\theta) = \{ f : \Sigma_\theta \to \C :\: f \text{ analytic and bounded} \} \] equipped with the uniform norm $\|f\|_{\infty,\theta}.$
Let further 
\[ \HI_0(\Sigma_\theta) = \bigl\{ f \in \HI(\Sigma_\theta):\: \exists \: C ,\epsilon > 0 :\: |f(z)| \leq C \min(|z|^\epsilon,|z|^{-\epsilon}) \bigr\}.\]
For a pseudo-$\omega$-sectorial operator $A$ and $\theta \in (\omega,\pi),$ one can define a functional calculus $\HI_0(\Sigma_\theta) \to B(X),\: f \mapsto f(A)$ extending the ad hoc rational calculus, by using a Cauchy integral formula.
If moreover, there exists a constant $C < \infty$ such that $\|f(A)\| \leq C \| f \|_{\infty,\theta},$ then $A$ is said to have bounded $\HI(\Sigma_\theta)$ calculus and if $A$ has dense range, the above functional calculus can be extended to a bounded Banach algebra homomorphism $\HI(\Sigma_\theta) \to B(X).$
This calculus also has the property $f_z(A) = T_z$ for $f_z(\lambda) = \exp(-z \lambda),\: z \in \Sigma_{\frac{\pi}{2} - \theta}.$

For further information on the $\HI$ calculus, we refer e.g. to \cite{KW04}. We now turn to H\"ormander function classes and their calculi.
\begin{defi}
\label{defi-Hoermander-class}
Let $\alpha > \frac12.$
We define the H\"ormander class by 
\[\Hor^\alpha_2 = \bigl\{ f : [0,\infty) \to \C \text{ is bounded and continuous on }(0,\infty), \:  \underbrace{|f(0)| + \sup_{R > 0} \| \phi f(R \,\cdot) \|_{W^\alpha_2(\R)}}_{=:\|f\|_{\Hor^\alpha_2}}< \infty \bigr\}.\]
Here $\phi$ is any $C^\infty_c(0,\infty)$ function different from the constant 0 function (different choices of functions $\phi$ resulting in equivalent norms) and $W^\alpha_2(\R)$ is the classical Sobolev space.
\end{defi}

The term $|f(0)|$ is not needed in the functional calculus applications of $\Hor^\alpha_2$ if $A$ is in addition injective.
The H\"ormander classes have the following properties.

\begin{lemma}
\label{lem-properties-Hoermander}
\begin{enumerate}
\item Assume that $\alpha \in \N$.
Then a locally integrable function $f : (0,\infty) \to \C$ belongs to the H\"ormander class $\Hor^\alpha_2$ if and only if
\[ |f(0)|^2 + \sum_{ k = 0}^\alpha \sup_{R > 0} \int_{R}^{2R} \Bigl|t^k \frac{d^k}{dt^k}f(t)\Bigr|^2 \,\frac{dt}{t} < \infty, \]
and the above quantity is equivalent to $\|f\|_{\Hor^\alpha_2}^2.$
\item $\Hor^\alpha_2$ is a Banach algebra for the pointwise multiplication.
\item Assume that $\frac12 < \alpha < \beta$.
Then $\Hor^\beta_2 \subseteq \Hor^\alpha_2$ is a continuous injection.
\end{enumerate}
\end{lemma}

\begin{proof}
See \cite[Section 3]{KrW3}, \cite[Section 4.2.1]{KrPhD} for the case that $\|f\|_{\Hor^\alpha_2}$ does not contain the summand $|f(0)|$.
The present case is deduced immediately.
\end{proof}

We can base a H\"ormander functional calculus on the $\HI$ calculus by the following procedure.

\begin{defi}
\label{defi-Hormander-calculus}
We say that a pseudo-$0$-sectorial operator has a bounded $\Hor^\alpha_2$ calculus if for some $\theta \in (0,\pi)$ and any $f \in \HI(\Sigma_\theta),$
$\|f(A)\| \leq C \|f\|_{\Hor^\alpha_2} ( \leq C' \|f\|_{\infty,\theta}).$
\end{defi}
In this case, the $\HI(\Sigma_\theta)$ calculus can be extended to a bounded Banach algebra homomorphism $\Hor^\alpha_2 \to B(X)$ in the following way.
Let 
\[ \W^\alpha_2 = \bigl\{ f : (0,\infty) \to \C :\: f \circ \exp \in W^\alpha_2(\R) \bigr\}\] equipped with the norm $\|f\|_{\W^\alpha_2} = \|f\circ \exp\|_{W^\alpha_2(\R)}.$
Note that for any $\theta \in (0,\pi),$ the space $\HI(\Sigma_\theta) \cap \W^\alpha_2$ is dense in $\W^\alpha_2$ \cite{KrW3}.
Since $\W^\alpha_2 \hookrightarrow \Hor^\alpha_2,$ by the above density, we get a bounded mapping $\W^\alpha_2 \to B(X)$ extending the $\HI$ calculus.

\begin{defi}
\label{defi-dyadic-partition}
Let $(\phi_k)_{k \in \Z}$ be a sequence of functions in $C^\infty_c(0,\infty)$ with the properties that $\supp \phi_k \subset [2^{k-1},2^{k+1}]$, $\phi_k(t) = \phi_0(2^{-k}t)$ and $\sum_{k \in \Z} \phi_k(t) = 1$ for all $t > 0.$
Then $(\phi_k)_{k \in \Z}$ is called a dyadic partition of unity.
\end{defi}

Let $(\phi_k)_{k \in \Z}$ be a dyadic partition of unity.
For $f \in \Hor^\alpha_2,$ we have that $\phi_k f \in \W^\alpha_2,$ hence $(\phi_k f)(A)$ is well-defined.
Then it can be shown that for any $x \in X,$ $\sum_{k= -n}^n (\phi_k f)(A) x$ converges as $n \to \infty$ and that it is independent of the choice of $(\phi_k)_{k \in \Z}.$
This defines the operator $f(A),$ which in turn yields a bounded Banach algebra homomorphism $\Hor^\alpha_2 \to B(X),\: f \mapsto f(A).$
This is the H\"ormander functional calculus.
For details of this procedure, we refer to \cite[Section 4]{KrW3}, \cite[Sections 4.2.3 - 4.2.6]{KrPhD}.
The H\"ormander functional calculus is the central object in this paper.
We shall access it by the following Theorem from \cite[Theorem 7.1]{KrW3} or \cite{KrPhD}.
Here and in what follows we put $\C_+ = \Sigma_{\frac{\pi}{2}} = \{ z \in \C :\: \Re (z) > 0 \}$.
For the notion of Pisier's property $(\alpha)$ we refer e.g. to \cite[4.9]{KW04}.

\begin{thm}
\label{thm-KrW3}
Let $A$ be a generator of an analytic semigroup on some reflexive Banach space $X$ with property $(\alpha)$.
Assume that $A$ has an $\HI(\Sigma_\sigma)$ calculus to some angle $\sigma \in (0,\pi)$ and that
\begin{equation}
\label{equ-R_T}
\bigl\{ \exp(-te^{i\theta}A):\: t > 0 \bigr\}
\end{equation}
is $R$-bounded in $B(X)$ with $R$-bound $\leq C (\cos(\theta))^{-\alpha}$.
This is clearly the case if
\begin{equation}
\label{equ-R_T-variant}
\bigl\{ \bigl(\cos(\arg z)\bigr)^\alpha \exp(-zA) :\: z \in \C_+ \bigr\}
\end{equation}
is $R$-bounded in $B(X)$.
Then $A$ has a H\"ormander functional calculus $\Hor^\beta_2$ on $X$ with $\beta > \alpha + \frac12$.
\end{thm}

\begin{proof}
The Theorem is proved in \cite[Theorem 7.1]{KrW3} under condition \eqref{equ-R_T-variant} and for $A$ having dense range.
An inspection of the proof given there shows that \eqref{equ-R_T} is sufficient, and the above alluded passage from the injective part $A_0$ to $A$ together with $|f(0)| \leq \|f\|_{\Hor^\beta_2}$ allows to conclude for general $A$.
\end{proof}

\subsection{UMD spaces, Banach lattices, $p$-convexity and $q$-concavity}
\label{subsec-UMD}

In this article, UMD lattices, i.e. Banach lattices which enjoy the UMD property, play a prevalent role.
For a general treatment of Banach lattices and their geometric properties, we refer the reader to \cite[Chapter 1]{LTz}.
We recall now definitions and some useful properties.
A Banach space $Y$ is called UMD space if the Hilbert transform 
\[ H : L^p(\R) \to L^p(\R),\: Hf(x) = \lim_{\epsilon \to 0} \int_{|x-y| \geq \epsilon} \frac{1}{x-y} f(y) \,dy \] 
extends to a bounded operator on $L^p(\R;Y),$ for some (equivalently for all) $1 < p < \infty$ \cite[Theorem 5.1]{HvNVW}.
The importance of the UMD property in harmonic analysis was recognized for the first time by Burkholder \cite{Burk1981,Burk1983}, see also his survey \cite{Burk2001}.
He settled a geometric characterization via a convex functional \cite{Burk1981} and together with Bourgain \cite{Bourgain1983}, they showed that the UMD property can be expressed by boundedness of $Y$-valued martingale sequences.
A UMD space is super-reflexive \cite{Al79}, and hence (almost by definition) B-convex.

A K\"othe function space $Y$ is a Banach lattice consisting of equivalence classes of locally integrable functions on some $\sigma$-finite measure space $(\Omega',\mu')$ with the additional properties
\begin{enumerate}
\item If $f :\: \Omega' \to \C$ is measurable and $g \in Y$ is such that $|f(\omega')| \leq |g(\omega')|$ for almost every $\omega' \in \Omega'$, then $f \in Y$ and $\|f\|_Y \leq \|g\|_Y$.
\item The indicator function $1_A$ is in $Y$ whenever $\mu'(A) < \infty$.
\item Moreover, we will assume that $Y$ has the Fatou property:
If a sequence $(f_k)_k$ of non-negative functions in $Y$ satisfies $f_k(\omega') \nearrow f(\omega')$ for almost every $\omega' \in \Omega'$ and $\sup_k \|f_k\|_Y < \infty$, then $f \in Y$ and $\|f\|_Y = \lim_k \|f_k\|_Y$.
\end{enumerate}
Note that for example, any $L^p(\Omega')$ space with $1 \leq p \leq \infty$ is such a K\"othe function space.

\begin{ass}
In the rest of the paper, $Y = Y(\Omega')$ will always be a UMD space which is also a K\"othe function space, unless otherwise stated.
\end{ass}

By $B$-convexity, $Y$ is order continuous and therefore $Y$ and its dual $Y'$ can be represented on the same measure space $(\Omega',\mu'),$ and moreover the duality is given simply by
\[ \langle y , y' \rangle = \int_{\Omega'} y(\omega') y'(\omega') \,d\mu'(\omega'), \]
see \cite[1.a, 1.b]{LTz}.
It is not difficult to show that if $Y$ is UMD, then also its dual is UMD.
Hence the dual of a UMD lattice is again a UMD lattice.
$L^p(\Omega;Y)$ is reflexive for $(\Omega,\mu)$ a $\sigma$-finite measure space, $Y$ a UMD space and $1 < p <\infty$, since $Y$ is reflexive and thus has the Radon-Nikodym property.
As a survey for UMD lattices and their properties in connection with results in harmonic analysis, we refer the reader to \cite{RdF}.

Let $Y = Y(\Omega')$ be a B-convex Banach lattice and $(\epsilon_k)_k$ an i.i.d. Rademacher sequence.
Then we have the norm equivalence
\begin{equation}
\label{equ-Rademacher-square}
\E \biggl\|\sum_{ k = 1}^n \epsilon_k y_k \biggr\|_Y \cong \biggl\|\Bigl( \sum_{k = 1}^n |y_k|^2 \Bigr)^{\frac12}\biggr\|_Y
\end{equation}
uniformly in $n \in \N$ \cite{Ma74}.
In particular, this also applies to $L^p(\Omega;Y),\: 1 < p < \infty,$ since this will also be a B-convex Banach lattice.

Let $E$ be any Banach space.
We can consider the vector valued lattice $Y(E) = \{ F : \Omega' \to E :\: F \text{ is strongly measurable and }\omega' \mapsto \|F(\omega')\|_E \in Y\}$ with norm $\|F\|_{Y(E)} = \bigl\| \|F(\cdot)\|_E \bigr\|$.
If $Y$ and $E$ are Banach lattices that are K\"othe function spaces with a $\sigma$-order continuous norm (e.g. an $L^p$ space for $p < \infty$ \cite[p.~235]{Lind} or more generally, a UMD lattice), then $[Y(E)]' = Y'(E')$ \cite[p.~239, p.~237]{Lind}.
From \cite[Corollary p.~214]{RdF}, we know that if $Y$ is UMD and $E$ is UMD, then also $Y(E)$ is UMD. Recall now the definition of $p$-convexity and $q$-concavity. 

\begin{defi}
\label{defi-p-convex}
Let $Y$ be a Banach lattice and $1 \leq p,q \leq \infty$.
Then $Y$ is called $p$-convex if there exists a constant $C < \infty$ such that for any $x_1,\ldots,x_n \in Y$, we have
\[ \Biggl\| \biggl( \sum_{i = 1}^n |x_i|^p \biggr)^{\frac1p} \Biggr\|_Y \leq C \biggl( \sum_{i = 1}^n \|x_i\|_Y^p \biggr)^{\frac1p} . \]
Similarly, $Y$ is called $q$-concave if there exists a constant $C > 0$ such that for any $x_1,\ldots,x_n \in Y$, we have
\[ \Biggl\| \biggl( \sum_{i = 1}^n |x_i|^q \biggr)^{\frac1q} \Biggr\|_Y \geq C \biggl( \sum_{i = 1}^n \|x_i\|_Y^q \biggr)^{\frac1q} . \]
\end{defi}

We now state several lemmas which will be relevant for the sequel.

\begin{lemma}
\label{lem-p-convexification}
Let $Y = Y(\Omega')$ be a Banach lattice and $1 \leq p < \infty$.
Suppose that $Y$ is $p$-convex.
Then $Y^p = \{ z : \: \Omega' \to \C :\: z\text{ measurable and there exists some }y\in Y :\: |z| = y^p \}$ together with the order induced by $Y$ and the norm $\|z\|_{Y^p} = \bigl\|\:|z|^{\frac1p}\bigr\|_Y^p$ is a Banach lattice.
$Y^p$ is called the $p$-convexification of $Y$.
\end{lemma}

\begin{proof}
See \cite[(c.2)]{RdF}.
\end{proof}

Note that we clearly have for $1\leq p \leq q < \infty$ that $Y^q = (Y^p)^{q/p}$.

\begin{lemma}
\label{lem-Tomczak}
Let $Y(\Omega') = [Y_0(\Omega'),L^2(\Omega')]_\theta$ be the complex interpolation space between a Banach lattice over $\Omega'$ and a Hilbert space, and $\theta \in (0,1)$.
Then $Y(\Omega')$ is $p$-convex and $q$-concave for $\frac1p = (1-\theta) + \frac{\theta}{2}$ and $\frac1q = \frac{\theta}{2}$.
Conversely, any $p$-convex and $q$-concave Banach lattice $Y(\Omega')$ with values of $p$ and $q$ as above for some $\theta \in (0,1)$ is, after renorming, of the form above.
\end{lemma}

\begin{proof}
Note Calderon's complex interpolation identity $[Y_0,L^2]_\theta = Y_0^{1-\theta} (L^2)^\theta = \{ y : \Omega' \to \C : \: y \text{ measurable and }|y| = |y_0|^{1-\theta} |y_1|^\theta \text{ for some }y_0 \in Y_0,\: y_1 \in L^2\}$ \cite[(c.3)]{RdF}.
Then the Lemma follows from the description of $p$-convex and $q$-concave lattices \cite[p.~218-219, Theorem 28.1]{TJ} after a possible renorming to have convexity and concavity constants equal to $1$ (see e.g. \cite[Proposition 3.3.6]{Lor}), together with the fact that a Hilbert space is both $2$-convex and $2$-concave.
\end{proof}

\begin{lemma}
\label{lem-Ls-convexity}
Let $1 \leq p \leq s \leq \infty$ with $p < \infty$.
Then $L^s(\Omega')$ is $s$-convex and $s$-concave.
Its $p$-convexification identifies to $L^s(\Omega')^p = L^{\frac{s}{p}}(\Omega')$.
\end{lemma}

\begin{proof}
We clearly have $\begin{displaystyle}\Bigl\| \Bigl( \sum_i |x_i|^s \Bigr)^{\frac1s} \Bigr\|_{L^s(\Omega')} = \Bigl( \sum_i \|x_i\|_{L^s(\Omega')}^s \Bigr)^{\frac1s}\end{displaystyle}$, which immediately gives $s$-convexity and $s$-concavity.
Moreover, $\begin{displaystyle}\|x\|_{L^s(\Omega')^p} = \bigl\|\:|x|^{\frac1p}\bigr\|_{L^s(\Omega')}^p = \|x\|_{L^{\frac{s}{p}}(\Omega')}\end{displaystyle}$, which shows the second statement.
\end{proof}

\begin{lemma}
\label{lem-UMD-lattice-convexity}
Let $p \in [1,2)$.
\begin{enumerate}
\item If $Y$ is a $p$-convex UMD lattice, then $Y(\ell^2)$ is also $p$-convex.
\item If $Y$ is a $p$-convex UMD lattice and $Y^p$ is also UMD, then $Y(\ell^2)^p$ is UMD.
\item $Y$ is a $p$-convex Banach lattice if and only if $Y'$ is a $p'$-concave Banach lattice.
\item If $Y$ is a $p$-convex Banach lattice, then it is also a $q$-convex Banach lattice for any $q \in (0,p)$.
\item If $Y$ is a UMD lattice, then there exists some $\epsilon > 0$ such that for any $0 < q < 1 + \epsilon$, $Y^q$ is a UMD lattice.
\end{enumerate}
\end{lemma}

\begin{proof}
1. Note that since $p \leq 2$, we have $\ell^p(\ell^2) \hookrightarrow \ell^2(\ell^p)$ (contractively).
This implies that if $y_k^i \in Y$, then pointwise 
\[ \Biggl\{ \sum_i \biggl( \sum_k |y_k^i(\omega')|^p \biggr)^{\frac{2}{p}} \Biggr\}^{\frac12} \leq \Biggl\{ \sum_k \biggl( \sum_i |y_k^i(\omega')|^2 \biggr)^{\frac{p}{2}} \Biggr\}^{\frac{1}{p}} . \]
Since $Y$ is a lattice, this implies
\begin{align*}
\Biggl\| \Biggl\{ \sum_i \biggl( \sum_k |y_k^i|^p \biggr)^{\frac{2}{p}} \Biggr\}^{\frac12} \Biggr\|_Y & \leq \Biggl\| \Biggl\{ \sum_k \biggl( \sum_i |y_k^i|^2 \biggr)^{\frac{p}{2}} \Biggr\}^{\frac{1}{p}} \Biggr\|_Y \\
& \leq \Biggl\{ \sum_k \Biggl\|  \biggl( \sum_i |y_k^i|^2 \biggr)^{\frac12} \Biggr\|_Y^p \Biggr\}^{\frac1p},
\end{align*}
where we have used that $Y$ is $p$-convex in the last step.

2. Let $H$ be the Hilbert transform on $L^2(\R)$.
Let further $\sum_k f_k \otimes (z_i^k)_i \in L^2(\R) \otimes Y(\ell^2)^{p}$.
We calculate
\begin{align*}
\Biggl\| \Bigl(H \otimes \Id_{Y(\ell^2)^p} \Bigr)\biggl(\sum_k f_k \otimes (z_i^k)_i \biggr) \Biggr\|_{L^2(\R;Y(\ell^2)^p)}^2 & = \int_{\R} \Biggl\| \sum_k Hf_k(t) (z_i^k)_i \Biggr\|_{Y(\ell^2)^p}^2 dt \\
& = \int_{\R} \Biggl\| \biggl( \sum_i \Bigl| \sum_k Hf_k(t) z_i^k\Bigr|^{\frac{2}{p}} \biggr)^{\frac12} \Biggr\|_{Y}^{2p} dt \\
& = \int_{\R} \Biggl\| \biggl( \sum_i \Bigl| \sum_k Hf_k(t) z_i^k\Bigr|^{\frac{2}{p}} \biggr)^{\frac{p}{2}} \Biggr\|_{Y^p}^p dt.
\end{align*}
The latter is a norm in $L^2(\R;Y^p(\ell^{\frac{2}{p}}))$.
Since $Y^p$ is a UMD space and $\ell^{\frac{2}{p}}$ also is a UMD space (as $\frac{2}{p} > 1$), according to \cite[Corollary p.~214]{RdF}, $Y^p(\ell^{\frac{2}{p}})$ is also a UMD space.
Thus we can estimate the expression above by the same term without $H$, which shows that $Y(\ell^2)^p$ is UMD.

3. We refer to \cite[Proposition 1.d.4]{LTz}.

4. We refer to \cite[Remark 3.4.14 (1)]{Lin}.

5. This is proved in \cite[Theorem 4]{RdF}.
\end{proof}

\begin{remark}
\label{rem-RdF-Problem-3}
In \cite[Problem 3]{RdF}, Rubio de Francia asks the question: if $Y$ is a UMD lattice that is also $p$-convex, is then also $Y^q$ a UMD lattice for any $q < p$?
This problem seems to be open since 1986.
In several instances in this article, we need that for a $p$-convex UMD lattice $Y$, the convexification $Y^p$ is also UMD and have thus to assume the latter as well.
\end{remark}

We turn to vector valued $L^p$ spaces, which are the underlying Banach spaces at the center of interest in this article.
In what follows, we let $(\Omega,\mu)$ be a $\sigma$-finite measure space.
For later use, we record the following fact.

\begin{lemma}
\label{lem-property-alpha}
Let $Y$ be a UMD lattice and $p \in (1,\infty)$.
Then $L^p(\Omega;Y)$ has Pisier's property $(\alpha)$.
\end{lemma}

\begin{proof}
Since $Y$ is UMD, it has finite concavity, and so finite cotype \cite[Proposition 1.f.3]{LTz}.
Thus, also $L^p(\Omega;Y)$ has finite cotype \cite[Theorem 11.12]{DiJT}.
Then according to \cite[N 4.8 - 4.10]{KW04}, the Banach function space $L^p(\Omega;Y)$ has property $(\alpha)$.
\end{proof}

\subsection{Tensor extension of operators to vector valued $L^p$ spaces}
\label{subsec-tensor}

We recall some technical points on linear and sublinear operators acting on $L^p(\Omega)$ and their tensor extensions to $L^p(\Omega;Y)$, where $Y$ is a UMD lattice.
For an operator $T$ acting on $L^p(\Omega)$ with $1 < p < \infty$ and $Y$ any Banach space,
we can consider the tensor extension $T \otimes \Id_Y : L^p(\Omega) \otimes Y \to L^p(\Omega) \otimes Y$ defined by $(T \otimes \Id_Y)(\sum_{k = 1}^n f_k \otimes y_k) = \sum_{k = 1}^n Tf_k \otimes y_k$.
Since $L^p(\Omega) \otimes Y$ is dense in $L^p(\Omega;Y)$, $T \otimes \Id_Y$ extends to a bounded operator on $L^p(\Omega;Y)$ if and only if
\[ \Bigl\| \sum_{k = 1}^n Tf_k \otimes y_k \Bigr\|_{L^p(\Omega;Y)} \leq C \Bigl\| \sum_{k = 1}^n f_k \otimes y_k \Bigr\|_{L^p(\Omega;Y)} \]
for some $C < \infty$ and any $\sum_{k = 1}^n f_k \otimes y_k \in L^p(\Omega) \otimes Y$.
We denote such an extension by slight abuse of notation again by $T$.

\begin{defi}
For $D \subset L^p(\Omega;Y)$ a subspace ($D = L^p(\Omega) \otimes Y$ mainly), $T : D \to L^p(\Omega;Y)$ is called sublinear if
\begin{align*}
|T(cf)| & = |c| \: |Tf| \\
|T(f+g)| & \leq |Tf| + |Tg| \\
|T(f) - T(g)| & \leq |T(f-g)|
\end{align*}
for any $f,g \in D$ and $c \in \C$.
\end{defi}

\begin{lemma}
\label{lem-density-sublinear}
Let $D$ above be a dense subspace of $L^p(\Omega;Y)$ and $T : D \to L^p(\Omega;Y)$ be a sublinear operator.
Assume that $T$ is bounded, that is, $\|Tf\|_{L^p(\Omega;Y)} \leq C \|f\|_{L^p(\Omega;Y)}$ for any $f \in D$.
Then $T$ extends uniquely to a bounded sublinear operator $T : L^p(\Omega;Y) \to L^p(\Omega;Y)$.
\end{lemma}

\begin{proof}
For $f \in L^p(\Omega;Y)$, let $(f_n)_n$ be an approximating sequence in $D$.
Define $Tf = \lim_n Tf_n$.
Hereby, by sublinearity of $T$, we have $\|Tf_n - Tf_m\| \leq \|T(f_n - f_m)\| \leq C \|f_n - f_m\|$, so that $Tf_n$ is indeed a Cauchy sequence in $L^p(\Omega;Y)$ and the limit does not change if the approximating sequence is changed.
Hence, $T : L^p(\Omega;Y) \to L^p(\Omega;Y)$ is well-defined.
Now boundedness is easy to show.
\end{proof}

In a similar manner, one can prove the following variant of the above lemma.

\begin{lemma}
Let $\tau$ be a family of bounded sublinear operators on $L^p(\Omega;Y)$.
Let $D$ be a dense subspace of $L^p(\Omega;Y)$.
\begin{enumerate}
\item Suppose that $\tau$ is $R$-bounded $D \to L^p(\Omega;Y)$, that is, there exists a constant $C < \infty$ such that for any $T_1,T_2,\ldots,T_N \in \tau$ and $f_1,\ldots,f_N \in D$, we have
\[ \Biggl\| \biggl( \sum_n |T_n f_n|^2 \biggr)^{\frac12} \Biggr\|_{L^p(\Omega;Y)} \leq C \Biggl\| \biggl( \sum_n |f_n|^2 \biggr)^{\frac12} \Biggr\|_{L^p(\Omega;Y)}. \]
Then $\tau$ is $R$-bounded $L^p(\Omega;Y) \to L^p(\Omega;Y)$.
\item Suppose that $\tau$ is lower $R$-bounded $D \to L^p(\Omega;Y)$.
Then $\tau$ is lower $R$-bounded $L^p(\Omega;Y) \to L^p(\Omega;Y)$.
\end{enumerate}
\end{lemma}

We will close the preliminaries with the following  subsection, which deals with space of homogeneous type and (generalised) Gaussian estimates.

\subsection{Spaces of homogeneous type, (generalised) Gaussian estimates}
\label{subsec-GE}

Let us first recall the definition of space of homogeneous type.

\begin{defi}
Let $(\Omega,\dist,\mu)$ be a metric measure space, that is, $\dist$ is a metric on $\Omega$ and $\mu$ is a Borel measure on $\Omega$.
We denote $B(x,r) = \{ y \in \Omega :\: \dist(x,y) \leq r \}$ the closed balls of $\Omega$.
We assume that $\mu(B(x,r)) \in (0,\infty)$ for any $x \in \Omega$ and $r > 0$.
Then $\Omega$ is said to be a space of homogeneous type if there exists a constant $C < \infty$ such that the doubling condition holds:
\[ \mu(B(x,2r)) \leq C \mu(B(x,r)) \quad (x \in \Omega,\: r > 0) . \]
\end{defi}

We write in short $V(x,r) = \mu(B(x,r))$.
In what follows, $(\Omega,\dist,\mu)$ is always a space of homogeneous type.
It is well-known that there exists some finite $d \in (0,\infty)$ such that $V(x,\lambda r) \leq C \lambda^d V(x,r)$ for any $x \in \Omega$, $r > 0$ and $\lambda \geq 1$.
Such a $d$ is called (homogeneous) dimension of $\Omega$.

\begin{lemma}
\label{lem-volume-homogeneous-type}
Let $(\Omega,\dist,\mu)$ be a space of homogeneous type.
Then there exists a constant $C > 0$ such that for all $r > 0$ and $x,y \in \Omega$ with $\dist(x,y) \leq r:$
\[ \frac{1}{C} V(x,r) \leq V(y,r) \leq C V(x,r) .\]
\end{lemma}

\begin{proof}
One has $B(x,r) \subseteq B(y,2r)$, so $V(x,r) \leq V(y,2r) \leq C V(y,r)$ according to the doubling condition.
The converse inequality is proved in the same way.
\end{proof}

We now introduce both the notions of  Gaussian estimates and generalised Gaussian estimates.

\begin{defi}
Let $(T_t)_{t \geq 0}$ be a semigroup acting on $L^2(\Omega)$.
Assume that \[T_tf(x) = \int_\Omega p_t(x,y) f(y) \,dy\] for any $f \in L^2(\Omega),\:x \in \Omega,\:t > 0$ and some measurable functions $p_t : \Omega \times \Omega \to \C$.
Let $m \geq 2$.
Then $(T_t)_t$ is said to satisfy Gaussian estimates (of order $m$) if there exist constants $C,c > 0$ such that
\begin{equation}
\label{equ-GE-prelims}
|p_t(x,y)| \leq C \frac{1}{V(x,r_t)} \exp\Biggl(-c \biggl(\frac{\dist(x,y)}{r_t}\biggr)^{\frac{m}{m-1}} \Biggr) \quad (x,y \in \Omega,\: t > 0),
\end{equation}
where $r_t = t^{\frac1m}$.
\end{defi}

\begin{defi}
Let $(\Omega,\dist,\mu)$ be a space of homogeneous type.
Let $A$ be a self-adjoint operator on $L^2(\Omega)$ generating the semigroup $(T_t)_{t \geq 0}$.
Let $p_0 \in [1,2)$ and $m \in [2,\infty).$
We say that $(T_t)_{t \geq 0}$ satisfies generalised Gaussian estimates (with parameters $p_0,m$) if there exist $c,C < \infty$ such that
\begin{multline}
\label{equ-GGE}
\bigl\| 1_{B(x,r_t)} T_t 1_{B(y,r_t)} \bigr\|_{L^{p_0}(\Omega) \to L^{p_0'}(\Omega)}\\  \leq C |V(x,r_t)|^{-(\frac{1}{p_0} - \frac{1}{p_0'})} \exp \Biggl(-c \biggl(\frac{\dist(x,y)}{r_t} \biggr)^{\frac{m}{m-1}}\Biggr) \quad (x,y \in \Omega,\: t > 0),
\end{multline}
where $r_t = t^{\frac1m}.$
\end{defi}

\begin{remark}
\label{rem-GE-GGE}
According to \cite[Proposition 2.9]{BK02} and \cite[Proposition 2.1]{BK05}, Gaussian estimates \eqref{equ-GE-prelims} with parameter $m \geq 2$ for a semigroup imply generalised Gaussian estimates \eqref{equ-GGE} with parameter $p_0 = 1$ and $m$.
Moreover, according to \cite[Proposition 2.1]{BK05}, generalised Gaussian estimates with parameters $p_0 \in [1,2)$ and $m \geq 2$ imply generalised Gaussian estimates with parameters $p_1 \in [p_0,2)$ and $m$.
\end{remark}

\begin{remark}
\label{rem-GGE-spectrum}
Assume that a semigroup is self-adjoint and satisfies generalised Gaussian estimates \eqref{equ-GGE}.
Then according to \cite[Proposition 2.1 (1) $\Longrightarrow$ (2) with $u=v=2$]{BK02}, we have
\[ \|1_{B_1} T_t 1_{B_2}\|_{2 \to 2} \leq g(\dist(B_1,B_2) r^{-1}) \]
for any ball $B_1,B_2 \subseteq \Omega$, with $g$ some bounded decreasing function.
This implies in particular that $\sup_{t \geq 0} \|T_t\|_{2 \to 2} < \infty$, which in turn gives $\sigma(A) \subseteq [0,\infty)$.
Thus any of our self-adjoint generators $A$ of the semigroup $T_t$ satisfying generalised Gaussian estimates is positive.
\end{remark}

In our work, we will make use of the fact that a space of homogeneous type $\Omega$ can be partitioned into finer and finer subsets which take over the role of dyadic cubes in $\R^d$.
This is the content of the following theorem.

\begin{thm}
\label{thm-dyadic-cubes}
Let $(\Omega, \dist, \mu)$ be a space of homogeneous type.
One can construct a \emph{dyadic system} 
$\mathcal{D} = \bigcup_{k\in\Z} \mathcal{D}_k$,
where each collection $\mathcal{D}_k$ consists of pairwise disjoint
sets of positive measure, \emph{dyadic cubes}, with the following properties:
\begin{itemize}
  \item $\Omega = \bigcup_{Q \in \mathcal{D}_k} Q$,
  \item if $Q\in\mathcal{D}_k$ and $R\in\mathcal{D}_l$ with $l\geq k$,
  then either $R\subset Q$ or $Q\cap R = \emptyset$,
  \item for some scaling parameter $\delta \in (0,1)$ it holds that, every 
  $Q\in\mathcal{D}_k$ contains a point $z$ for which
  \begin{equation*}
    B(z, \delta^k/3) \subset Q \subset B(z, 2\delta^k) .
  \end{equation*}
\end{itemize}
\end{thm}
\begin{proof}
  Theorem 2.2 in \cite{HyKa} with $\delta \leq 1/12$, $c_0 = C_0 = 1$,
  $c_1 = 1/3$, $C_1 = 2$, together with \cite[2.21]{HyKa}.
\end{proof}

According to the next two lemmas, arbitrary balls in $\Omega$ are comparable to these dyadic cubes in a certain sense.

\begin{lemma}
\label{lem-dyadic-cubes-lower}
Let $(\Omega,\dist,\mu)$ be a space of homogeneous type.
Let $\mathcal{D}$ be a dyadic system from Theorem \ref{thm-dyadic-cubes}.
For every $r>0$ there exists an integer $k(r)$ such that
if $x\in Q \in \mathcal{D}_{k(r)}$, then $Q\subset B(x,r)$ and
$\mu (B(x,r)) \lesssim \mu (Q)$ with implied constant independent of $x$ and $r$.
\end{lemma}

\begin{proof}
For $r>0$ let $k(r)$ be the smallest integer for which $\delta^{k(r)} < r/4$,
so that $\delta r \leq 4\delta^{k(r)} < r$.
For any $x\in\Omega$ and $r>0$, there exists a unique $Q\in\mathcal{D}_{k(r)}$
containing $x$. Now, for some $z\in Q$, since $\dist(x,z) \leq 2 \delta^{k(r)}$,
\begin{equation*}
  Q \subset B(z,2\delta^{k(r)}) \subset B(x,4\delta^{k(r)}) \subset B(x,r) .
\end{equation*}
On the other hand, $r\leq 4\delta^{k(r) - 1} = 12\delta^{-1} (\delta^{k(r)} / 3)$
so that
\begin{equation*}
  \mu (B(x,r)) \lesssim \mu (B(z,r)) \lesssim \mu (B(z,\delta^{k(r)} / 3))
  \leq \mu (Q),
\end{equation*}
where we have used Lemma \ref{lem-volume-homogeneous-type} together with $\dist(x,z) \leq 2 \delta^{k(r)} \leq r$.
\end{proof}

\begin{lemma}
\label{lem-systems}
  Let $(\Omega, \mu)$ be a space of homogeneous type.
  There exists a finite collection of dyadic systems 
  $\mathcal{D}^m$ on $\Omega$, so that for every ball $B$
  one can find a dyadic cube $Q_B$ in one of the systems such that
  $B\subset Q_B$ and $\mu(Q_B) \lesssim \mu(B)$.
\end{lemma}

\begin{proof}
See \cite[2.21 and Theorem 4.1]{HyKa}.
\end{proof}

\section{The Hardy-Littlewood lattice maximal function}
\label{sec-MHL}

In this section we consider the Hardy-Littlewood lattice maximal
function
\begin{equation}
  M_{HL}(f)(x,\omega') = \sup_{r>0} \frac{1}{V(x,r)}
  \int_{B(x,r)} |f(y,\omega')| \, d\mu(y), \quad x\in\Omega, \quad
  \omega'\in\Omega' ,
\end{equation}
and prove the following result:

\begin{thm}
\label{thm-MHL}
  $M_{HL}$ is bounded on $L^p(\Omega;Y)$ for any $p \in (1,\infty)$ and for every UMD lattice $Y$.
\end{thm}

Boundedness of such vector-valued maximal operators originates in the case $Y = \ell^q$ and $\Omega = \R^d$ in the work of Fefferman and Stein \cite{FeSt}.
The boundedness of lattice maximal operators is commonly abstracted
in the following Banach space property \cite{GMT1,GMT2}:

\begin{defi}
Let $Y=Y(\Omega')$ be a Banach lattice and denote 
by $\mathcal{D}$ the family
of dyadic intervals on the unit interval $[0,1)$.
The space $Y$ is said to have the \emph{Hardy-Littlewood property}
if the dyadic lattice maximal function
\begin{equation}
  M_d(f)(x,\omega') = \sup_{\substack{I\ni x \\ I\in\mathcal{D}}}
  \frac{1}{|I|} \biggl| \int_I f(y,\omega') \, dy \biggr| , \quad
  x\in [0,1), \quad \omega'\in\Omega' ,
\end{equation}
defines a bounded operator on $L^p([0,1); Y)$ 
for one (or, equivalently, for all) $p\in (1,\infty)$.
\end{defi}

Note that the definition in \cite{GMT1} refers directly to $M_{HL}$ on $\R^d$.
The point of this section is to extend this property to $M_{HL}$ on spaces of
homogeneous type, and it is useful to begin with dyadic maximal operators. 
It is immediately clear that $M_{HL}$ dominates any dyadic maximal operator 
on $\R^d$. Conversely, using the well-known Euclidean version of 
Lemma \ref{lem-systems} above, we see that dyadic maximal operators dominate $M_{HL}$. The equivalence between the definition in
\cite{GMT1} and the one above will thereby quickly follow from our considerations.

The UMD property was connected with the Hardy-Littlewood property
by Bourgain in \cite[Lemma 1]{Bou84}, see also \cite[Theorem 3]{RdF}.

\begin{thm}[Bourgain]
\label{thm-Bourgain}
  Let $Y$ be a Banach lattice. Then $Y$ is UMD if and only
  if $Y$ and $Y'$ have the Hardy-Littlewood property.
\end{thm}

The proof of Theorem \ref{thm-MHL} is based on the following
transference result:

\begin{lemma}
\label{lemma-MHL}
  Let $Y = Y(\Omega')$ be a Banach lattice. Further, 
  let $\mathcal{F} = (\mathcal{F}_k)_{k\in\Z}$
  be a filtration on $(\Omega, \mu)$
  and denote by $E_k$ the corresponding conditional expectation
  operators. If $Y$ has the Hardy-Littlewood property, then
  the lattice maximal function
  \begin{equation}
    M_{\mathcal{F}}(f)(x,\omega') = \sup_{k\in\Z} |E_kf(x,\omega')| ,
    \quad x\in\Omega, \quad \omega'\in\Omega' ,
  \end{equation}
  defines a bounded operator on $L^p(\Omega ; Y)$ for all 
  $p\in (1,\infty)$. Moreover, the operator norm of
  $M_\mathcal{F}$ is not greater than the operator norm of $M_d$.
\end{lemma}

The proof of Lemma \ref{lemma-MHL} is based on a concave function
argument originating from the work of Burkholder \cite{Burk1981} (see also \cite{Burk2001}). We follow closely the argument presented in \cite[Section 7]{Kempp1}.
We begin by observing that, given a filtration 
$(\mathcal{F}_k)_{k\in\N}$ on
$(\Omega, \mu)$, the inequality
\begin{equation}
  \int_\Omega \bigl\| \sup_{0\leq k \leq n} |E_kf(x,\cdot)|
  \bigr\|^p \, d\mu (x) \leq C \int_\Omega \| E_nf(x) \|^p \,
  d\mu(x), \quad f\in L^p(\Omega; Y), \quad n \in\N ,
\end{equation}
where $C$ is a fixed constant, is equivalent with
\begin{equation}
\label{eq-V}
\tag{$\ast$}
  \int_\Omega V_p \Big( \{ E_kf(x) \}_{k=0}^n , E_nf(x) \Big)
  \, d\mu(x) \leq 0 , \quad f\in L^p(\Omega, Y), \quad n\in\N ,
\end{equation}
where
\begin{equation}
  V_p(S,y) = \bigl\| \sup_{y'\in S} |y'(\cdot)| \bigr\|^p - C \| y \|^p,
  \quad S\subset Y \textup{ finite}, \quad y\in Y.
\end{equation}

\begin{lemma}
\label{lem-concave-U}
  Suppose that $Y$ is a Banach lattice and let
  $1 < p < \infty$. The following conditions are equivalent:
  \begin{enumerate}
    \item \eqref{eq-V} holds for the dyadic filtration on the unit interval (with the Lebesgue measure).
    \item There exists a real-valued function $U : \{ \text{finite subsets of }Y \} \times Y \to \R$ such that
    \begin{itemize}
      \item $U(S,y) \geq V_p(S,y)$
      \item $U(S\cup \{ y \}, y) = U(S,y)$
      \item $U(\emptyset, y) \leq 0$
      \item $U(S, \cdot)$ is concave
    \end{itemize}
    for finite subsets $S$ of $Y$ and $y\in Y$.
    \item \eqref{eq-V} holds for any filtration on any
    $\sigma$-finite measure space.
  \end{enumerate}
\end{lemma}
\begin{proof}
  \textbf{1. $\Rightarrow$ 2.}
  We define
  \begin{multline}
    U(S,y) = \sup \Biggl\{ \int_0^1 V_p \Bigl( S \cup \{ E_kf(x) \}_{k=0}^n , E_nf(x) \Bigr) \, dx :\\  \int_0^1 f(x)\, dx = y, \: f \: \mathcal{F}_n\text{-measurable and }Y\text{-valued, }n\in\N
    \Biggr\}
  \end{multline}
  for finite $S\subset Y$ and $y\in Y$.
  
  That $U(S,y) \geq V_p(S,y)$ is immediate from the definition of
  $U$ already with $n=0$. Likewise, that
  $U(S\cup \{ y \}, y) = U(S,y)$ follows at once from the
  observation that $\{ y \} \subset \{ E_kf(x) \}_{k=0}^n$
  for almost every $x\in [0,1)$ whenever $\int_0^1 f(x)\, dx = y$
  and $n\in\N$. That $U(\emptyset , y) \leq 0$ is exactly the
  assumption $\textit{1.}$
  
  To see that $U(S, \cdot)$ is concave, we first show that it 
  is midpoint concave, i.e. that for any $y_1, y_2 \in Y$ we have
  \begin{equation}
  \label{eq-midpointconcave}
    U\Bigl(S, \frac{y_1 + y_2}{2}\Bigr) \geq \frac{1}{2} \bigl( U(S,y_1) + U(S,y_2)\bigr).
  \end{equation}
  
  To deal with the suprema, let $m_i < U(S,y_i)$
  for $i=1,2$. By the definition of $U$ there
  exist functions $f_1,f_2 \in L^p([0,1);Y)$
  for which $\int_0^1 f_i (x) \, dx = y_i$ and
  \begin{equation*}
    \int_0^1 V_p \Bigl( S \cup \{ E_kf_i(x) \}_{k=0}^n , E_nf_i(x) \Bigr) \, dx > m_i , \quad i=1,2 .
  \end{equation*}
  The function defined as
  \begin{equation*}
    f(x) = \begin{cases}
      f_1(2x), \quad &0\leq x < 1/2, \\
      f_2(2x-1), \quad &1/2\leq x < 1,
    \end{cases}
  \end{equation*}
  will then satisfy
  \begin{equation*}
    \int_0^1 V_p \Bigl( S \cup \{ E_kf(x) \}_{k=0}^{n+1} , E_{n+1}f(x) \Bigr) \, dx > \frac{m_1 + m_2}{2} . 
  \end{equation*}
  To see this note that
  \begin{equation*}
  \begin{split}
    \{ E_kf(x) \}_{k=1}^{n+1} \supset
    \{ E_kf_1(2x) \}_{k=0}^n , \quad 0\leq x < 1/2,\\
   \{ E_kf(x) \}_{k=1}^{n+1} \supset
    \{ E_kf_2(2x-1) \}_{k=0}^n , \quad 1/2\leq x < 1.   
  \end{split}
  \end{equation*}
  Therefore
  \begin{equation*}
  \begin{split}
    &\int_0^{1/2} V_p \Bigl( S \cup \{ E_kf(x) \}_{k=0}^{n+1} , E_{n+1}f(x) \Bigr) \, dx \\
    &\geq \int_0^{1/2} V_p \Bigl( S \cup \{ E_kf_1(2x) \}_{k=0}^n , E_nf_1(2x) \Bigr) \, dx \\
    &= \frac{1}{2} \int_0^1 V_p \Bigl( S \cup \{ E_kf_1(x) \}_{k=0}^n , E_nf_1(x) \Bigr) \, dx 
    > \frac{m_1}{2}
    \end{split}
  \end{equation*}
  and similarly
  \begin{equation*}
    \int_{1/2}^1 V_p \Bigl( S \cup \{ E_kf(x) \}_{k=0}^{n+1} , E_{n+1}f(x) \Bigr) \, dx
    > \frac{m_2}{2} .
  \end{equation*}
  Since $m_i$ were arbitrary, \eqref{eq-midpointconcave} follows.
  
  To finish this part of the proof, we remark that a midpoint concave function that is also locally
  bounded from below is actually concave.  
  We have now shown that $U$ satisfies the required conditions.
  
  \textbf{2. $\Rightarrow$ 3. (finite space, finite algebras)}
  For this step we first consider 
  filtrations of finite algebras and then reduce the general
  case to this. 
  
  Claim: If $(\mathcal{F}_k)_{k\in\N}$ is a
  filtration of finite algebras on a finite measure space
  $(\Omega, \mu)$, then
  \begin{equation}
    \int_\Omega U\Bigl( \{ E_kf(x) \}_{k=0}^n, E_nf(x) \Bigr) \, d\mu(x) \leq \int_\Omega U\Big( \{ E_kf(x) \}_{k=0}^{n-1}, E_{n-1}f(x) \Big) \, d\mu(x)
  \end{equation}
  for all $f\in L^1(\Omega; Y)$ and $n\in\N$.
  
  Proof of claim: By the second property in Lemma \ref{lem-concave-U} $\textit{2.}$,
  $\begin{displaystyle}U\Bigl( \{ E_kf(x) \}_{k=0}^n, E_nf(x) \Bigr) = 
  U\Bigl( \{ E_kf(x) \}_{k=0}^{n-1}, E_nf(x) \Bigr)\end{displaystyle}$.
  Moreover, on each generator $A$ of the (finite) algebra
  $\mathcal{F}_{n-1}$, the set $\{ E_kf(x) \}_{k=0}^{n-1}$ 
  is a constant $Y_A$ and   
  \begin{equation}
    \int_A U\bigl(Y_A, E_nf(x)\bigr) \, d\mu(x)
    \leq \mu(A) U \Bigl( Y_A, 
    \frac{1}{\mu(A)} \int_A E_nf(x) \, d\mu(x) \Bigr).
  \end{equation}  
  Therefore
  \begin{equation}
    \begin{split}
      &\int_\Omega U\Bigl( \{ E_kf(x) \}_{k=0}^n , E_nf(x) \Bigr)
      \, d\mu(x) \\ 
      &= \sum_{A\in \textup{gen}(\mathcal{F}_{n-1})}
      \int_A U\bigl(Y_A, E_nf(x)\bigr) \, d\mu(x) \\
      &\leq \sum_{A\in \textup{gen}(\mathcal{F}_{n-1})}
      \mu(A) \, U \Bigl( Y_A, \frac{1}{\mu(A)} \int_A E_nf(x) 
      \, d\mu(x) \Bigr) \\
      &= \sum_{A\in \textup{gen}(\mathcal{F}_{n-1})}
      \int_A U\bigl( Y_A, E_{n-1}f(x)\bigr) \, d\mu(x) \\
      &= \int_\Omega U\Bigl( \{ E_kf(x) \}_{k=0}^{n-1}, E_{n-1}f(x) \Bigr) \, d\mu(x).
    \end{split}
  \end{equation}
  
We have now shown: $\textit{2.}$ implies that the maximal
function
\begin{equation*}
  M_\mathcal{F}^{(n)}(f)(x, \omega')
  = \sup_{0 \leq k \leq n} | E_kf(x,\omega')|, \quad x\in\Omega, \quad \omega'\in\Omega',
\end{equation*}
satisfies
\begin{equation}
\label{eq:maximalineq}
  \| M_\mathcal{F}^{(n)}(f) \|_{L^p(\Omega; Y)}
  \leq C \| E_n f \|_{L^p(\Omega; Y)}, \quad 
  f\in L^p(\Omega;Y), \quad n\in\N,
\end{equation}
uniformly for all filtrations $(\mathcal{F}_k)_{k\in\N}$ of
finite algebras on any finite measure space $(\Omega, \mu)$.

\textbf{2. $\Rightarrow$ 3. (reduction to finite algebras)}
To see that the finiteness requirement for $\sigma$-algebras
is not necessary we argue as follows:

Suppose that $(\mathcal{F}_k)_{k\in\N}$ is a filtration, $n$ a positive integer and
$f$ a function in $L^p(\Omega ; Y)$. 
Let $\varepsilon > 0$ and begin by choosing simple functions 
$s_k \in L^p(\mathcal{F}_k ; Y)$, $k=0,1,\ldots ,n$,
so that
\begin{equation*}
  \| E_kf - s_k \|_{L^p(\Omega;Y)} < \frac{\varepsilon}{2^{k+2}} .
\end{equation*}
For $k=0,1,\ldots ,n$, let $\widetilde{\mathcal{F}}_k$ be the finite algebra generated by 
$s_0,s_1,\ldots , s_k$ and observe that $\widetilde{\mathcal{F}}_k \subset \mathcal{F}_k$ and that
$\widetilde{\mathcal{F}}_k \subset \widetilde{\mathcal{F}}_{k+1}$, i.e. that 
$(\widetilde{\mathcal{F}}_k)_{k=0}^n$ is a filtration.
    Now
    \begin{align*}
      \| M_\mathcal{F}^{(n)} (f) \|_{L^p(\Omega;Y)}
      &= \biggl( \int_\Omega
      \bigl\| \sup_{0\leq k \leq n} E_kf(x, \cdot ) \bigr\|^p \, d\mu (x) \biggr)^{1/p} \\
      &\leq \biggl( \int_\Omega\bigl\| \sup_{0\leq k \leq n} (E_kf(x,\cdot ) - \widetilde{E}_kf(x,\cdot )) \bigr\|^p \, d\mu (x) \biggr)^{1/p} \\
      &+ \| \widetilde{M}_\mathcal{F}^{(n)} (f) \|_{L^p(\Omega;Y)} ,
    \end{align*}
    where the maximal operator $\widetilde{M}_\mathcal{F}^{(n)}$ 
    satisfies 
    $\| \widetilde{M}_\mathcal{F}^{(n)} (f) \|_{L^p(\Omega;Y)} \leq C \| E_nf \|_{L^p(\Omega;Y)}$ for a constant $C$ independent
    of the filtration $(\widetilde{\mathcal{F}}_k)_{k=0}^n$. This independence is crucial, as 
    $\widetilde{\mathcal{F}}_k$'s arose from $f$.

    We then estimate
    \begin{align*}
      &\biggl( \int_\Omega \bigl\| \sup_{0\leq k \leq n} ( E_kf(x,\cdot )
      - \widetilde{E}_kf(x,\cdot )) \bigr\|^p \, d\mu (x) \biggr)^{1/p}\\
      &\leq \biggl( \int_\Omega \biggl( \sum_{k=0}^n \bigl\| E_kf(x,\cdot )
      - \widetilde{E}_kf(x,\cdot ) \bigl\| \biggr) ^p \, d\mu (x) \biggr)^{1/p} \\
      &\leq \sum_{k=0}^n \| E_kf - \widetilde{E}_kf \|_{L^p(\Omega;Y)} \\
      &\leq \sum_{k=0}^n \Bigl( \bigl\| E_kf - s_k \bigr\|_{L^p(\Omega;Y)} 
      + \| \widetilde{E}_kf - s_k \|_{L^p(\Omega;Y)} \Bigr) .     
    \end{align*}
    Furthermore, since 
    \begin{align*}
      \| \widetilde{E}_kf - s_k \|_{L^p(\Omega;Y)} &= \| \widetilde{E}_kf - \widetilde{E}_k s_k \|_{L^p(\Omega;Y)}
      = \| \widetilde{E}_k (E_kf 
      - s_k) \|_{L^p(\Omega;Y)} \\
      &\leq \| E_kf - s_k \|_{L^p(\Omega;Y)} ,
    \end{align*}
    we get
    \begin{align*}
      \biggl( \int_\Omega \bigl\| \sup_{0\leq k \leq n} ( E_kf(x,\cdot )
      - \widetilde{E}_kf(x,\cdot )) \bigr\|^p \, d\mu (x) \biggr)^{1/p}
      &\leq 2 \sum_{k=0}^n \| E_kf - s_k \|_{L^p(\Omega;Y)} \\
      &< \sum_{k=0}^n \frac{\varepsilon}{2^{k+1}} < \varepsilon .
    \end{align*}
  
\textbf{2. $\Rightarrow$ 3. (reduction to finite space)}  
As the final step, we will get rid of the assumption that
the measure space $(\Omega, \mu)$ is finite.

Suppose then that \eqref{eq:maximalineq} holds 
uniformly with respect to any 
filtration $\mathcal{F}$ on any 
finite measure space and let
$(\Omega, \mu )$ be a $\sigma$-finite measure space with a filtration $\mathcal{F} = (\mathcal{F}_k)_{k=0}^{\infty}$.
Since $\mathcal{F}_0$ is $\sigma$-finite (by assumption), 
we can write $\Omega$ as a union of disjoint sets $A_j\in\mathcal{F}_0$, 
$j\in\N$, each with finite $\mu$-measure. 
Let us define for $j\in\N$ the 
finite measures $\mu_j (A) = \mu (A \cap A_j)$ on $\Omega$.
The conditional expectation of a function $f\in L^p(\Omega ; Y)$ with
respect to $\mathcal{F}_k$ and $\mu_j$ is simply the conditional expectation of $1_{A_j}f$ with respect to
$\mathcal{F}_k$ which further equals $1_{A_j} E_k f$ (since $\mathcal{F}_0 \subset \mathcal{F}_k$ for all $k$). In symbols
\begin{equation*}
  E_k^{(j)} f = 1_{A_j} E_kf ,
\end{equation*}
where $E_k^{(j)} f$ denotes the conditional expectation of $f$ with respect to $\mathcal{F}_k$ and $\mu_j$.
Thus
\begin{align*}
  \| M_\mathcal{F}^{(n)}(f) \|_{L^p(\Omega;Y)}^p 
  &= \sum_{j=0}^{\infty} \int_{A_j} 
  \bigl\| \sup_{0\leq k \leq n} E_k f(x,\cdot ) \bigr\|^p \, d\mu(x) \\
  &= \sum_{j=0}^{\infty} \int_{A_j}  
  \bigl\| \sup_{0\leq k \leq n} E_k^{(j)} f(x,\cdot ) \bigr\|^p \, d\mu_j(x) \\
  &\leq \sum_{j=0}^{\infty} C^p \int_{A_j} \bigl\| E_n f(x,\cdot ) \bigr\|^p \, d\mu_j (x) \\
  &= C^p \| E_nf \|_{L^p(\Omega;Y)}^p .
\end{align*}

So far we have only considered filtrations indexed by $\N$. Suppose that \eqref{eq:maximalineq} holds with
respect to any filtration indexed by $\N$ on any $\sigma$-finite measure space and let
$\mathcal{F} = (\mathcal{F}_k)_{k\in\Z}$ be a filtration on $(\Omega, \mu )$. Then for all $N\geq 0$, \eqref{eq:maximalineq} holds with respect to $(\mathcal{F}_k)_{k=-N}^{\infty}$ with a constant
independent of $N$ and thus by monotone convergence theorem with respect to $(\mathcal{F}_k)_{k\in\Z}$.  
  
  This finishes the proof.
\end{proof}

Let now $\mathcal{D} = \bigcup_{k\in\Z} \mathcal{D}_k$ be a dyadic system on
$\Omega$ as in Theorem \ref{thm-dyadic-cubes}. Denote by $\mathcal{F}_k$
the $\sigma$-algebra generated by $\mathcal{D}_k$ and note that 
the corresponding conditional expectation is
\begin{equation}
\label{equ-conditional-expectation}
  E_kf(x) = \sum_{Q\in\mathcal{D}_k} \frac{1_Q(x)}{\mu(Q)}
  \int_Q f(y) \, d\mu(y) .
\end{equation}
The maximal function associated with the increasing filtration
$(\mathcal{F}_k)_{k\in\Z}$ is therefore given by
\begin{equation}
  M_\mathcal{F}(f)(x,\omega') = \sup_{\substack{Q\ni x \\ Q\in\mathcal{D}}} \frac{1}{\mu(Q)} \biggl| \int_Q f(y,\omega') \, d\mu(y)
  \biggr| , \quad x\in\Omega, \quad \omega'\in\Omega' .
\end{equation}

Lemma \ref{lem-systems} allows us to control the Hardy-Littlewood lattice maximal
function $M_{HL}$ by its dyadic counterparts. Indeed, we see that for any ball $B$,
\begin{equation}
  \frac{1}{\mu(B)} \int_B |f(y,\omega')| \, d\mu(y)
  \lesssim \frac{1}{\mu(Q_B)} \int_{Q_B} |f(y,\omega')| \, d\mu(y) .
\end{equation}
Therefore,
\begin{equation}
\label{eq-domination}
  M_{HL}(f)(x,\omega') \lesssim \sum_m M_{\mathcal{F}^m}(|f|)(x,\omega'),
  \quad x\in\Omega, \quad \omega'\in\Omega' , 
\end{equation}
where $\mathcal{F}^m$ are the filtrations arising from the
finite collection of dyadic systems $\mathcal{D}^m$.

We are now in a position to prove Theorem \ref{thm-MHL}.

\begin{proof}[of Theorem \ref{thm-MHL}]
  The result is an immediate consequence of the considerations
  above. Indeed, if $Y$ is a UMD lattice, it has the Hardy-Littlewood property by Theorem \ref{thm-Bourgain}. By Lemma
  \ref{lem-systems}, we may construct dyadic filtrations
  $\mathcal{F}^m$ so that $M_{HL}$ is dominated pointwise
  by the sum of $M_{\mathcal{F}^m}$ as in \eqref{eq-domination}.
  By Lemma \ref{lemma-MHL}, the latter maximal operators are
  bounded on $L^p(\Omega; Y)$ for any $p\in (1,\infty)$,
  and therefore so is $M_{HL}$.
\end{proof}

We remark that in the case $Y = \ell^s(\ell^2)$, Theorem \ref{thm-MHL} can be proved by a similar method as \cite[Theorem 1.2, Corollary 2.9]{GLY}.
Important for our later considerations will also be the following variant of the centered Hardy-Littlewood maximal operator.

\begin{defi}
\label{defi-local-Lq-average}
Let $f : \Omega \to Y$ locally integrable, $q \in [1,\infty]$ and $r > 0.$
\begin{enumerate}
\item We put
\begin{equation}
\label{equ-defi-Nqr}
N_{q,r}(f)(x,\omega') = \frac{1}{V(x,r)^{\frac{1}{q}}} \biggl(\int_{B(x,r)} |f(y,\omega')|^q \,d\mu(y) \biggr)^{\frac1q},
\end{equation}
(obvious modification if $q = \infty$).
\item Furthermore, we put
\begin{equation}
\label{equ-defi-Mq}
M_{HL}^q(f)(x,\omega') = \sup_{r > 0} N_{q,r}(f)(x,\omega').
\end{equation}
\end{enumerate}
\end{defi}

These operators are well-defined a priori on $L^p(\Omega) \otimes Y$, and they are sublinear on that subspace.
We will show below that they are bounded $L^p(\Omega) \otimes Y \subseteq L^p(\Omega;Y) \to L^p(\Omega;Y)$, so that by Lemma \ref{lem-density-sublinear}, they extend boundedly to $L^p(\Omega;Y)$.

\begin{prop}
\label{prop-MHL-q}
Let $Y$ be a $p_Y$-convex UMD lattice for some $p_Y \in [1,\infty]$ such that the convexification $Y^{p_Y}$ is again a UMD lattice, $q \in [1,\infty)$ and $p \in (1,\infty)$.
Assume that $p_Y \geq q$ and $p > q$.
Then $M_{HL}^q$ is bounded on $L^p(\Omega;Y)$.
\end{prop}

\begin{proof}
Let $f \in L^p(\Omega) \otimes Y \subseteq L^p(\Omega;Y)$.
For $x \in \Omega$ and $\omega' \in \Omega'$, we put $g(x,\omega') = |f(x,\omega')|^q$.
Then
\begin{align*}
\| M_{HL}^q(f) \|_{L^p(\Omega;Y)}^p & = \int_\Omega \Biggl\| \Biggl( \sup_{r > 0} \frac{1}{V(x,r)} \int_{B(x,r)} |f(y,\omega')|^q \,dy \Biggr)^{\frac1q} \Biggr\|_Y^p \,dx \\
& = \int_\Omega \Biggl\| \Biggl( \sup_{r > 0} \frac{1}{V(x,r)} \int_{B(x,r)} g(y,\omega') \,dy \Biggr)^{\frac1q} \Biggr\|_Y^p \,dx \\
& = \int_\Omega \Biggl\| \sup_{r > 0} \frac{1}{V(x,r)} \int_{B(x,r)} g(y,\omega') \,dy \Biggr\|_{Y^q}^{\frac{p}{q}} \,dx \\
& \lesssim \int_\Omega \bigl\| g(x,\omega') \bigr\|_{Y^q}^{\frac{p}{q}} \,dx \\
& = \int_\Omega \bigl\| f(x,\omega') \bigr\|_Y^p \,dx = \|f\|_{L^p(\Omega;Y)}^p,
\end{align*}
where we have used that $\frac{p}{q} > 1$ and that $Y^{p_Y}$ is a UMD lattice, so also $Y^q = (Y^{p_Y})^{\frac{q}{p_Y}}$ is a UMD lattice according to Lemma \ref{lem-UMD-lattice-convexity}.
Hence Theorem \ref{thm-MHL} was applicable on $L^{\frac{p}{q}}(\Omega;Y^q)$.
\end{proof}

\section{H\"ormander functional calculus}
\label{sec-main}

In this section, we prove the H\"ormander functional calculus result on $L^p(\Omega;Y)$, using the $R$-boundedness of the semigroup for complex times.
The main idea is to estimate the semigroup against the Hardy-Littlewood maximal operator.
To push down the H\"ormander calculus derivation exponent and also to treat generalised Gaussian estimates in place of classical Gaussian estimates,
we will the need the local $L^q$ average operator $N_q$ given in Definition \ref{defi-local-Lq-average}.
We will also illustrate in this section several consequences of the H\"ormander functional calculus result.
In the following definition, we give the parameter needed in the H\"ormander calculus, which will encode convexity and concavity of $Y$ and the Lebesgue $L^p$ exponent.

\begin{defi}
Let $p \in (1,\infty),$ $p_Y \in (1,2]$ and $q_Y \in [2,\infty).$
We put 
\begin{equation}
\label{equ-defi-alpha}
\alpha(p,p_Y,q_Y) = \max\biggl(\frac{1}{p},\frac{1}{p_Y},\frac12 \biggr) - \min \biggl(\frac{1}{p},\frac{1}{q_Y},\frac12 \biggr) \in (0,1).
\end{equation}
Informally spoken, this is the length of the segment, which is the convex hull of the points $\frac{1}{p},\frac{1}{p_Y},\frac{1}{q_Y}$ and $\frac12$ sitting on the real line.
\end{defi}

The $R$-boundedness of the semigroup under generalised Gaussian estimates reads as follows.

\begin{thm}
\label{thm-R-bounded-semigroup}
Let $(\Omega,\dist,\mu)$ be a space of homogeneous type with a dimension $d$.
Let $A$ be a self-adjoint operator on $L^2(\Omega)$ generating the semigroup $(T_t)_{t \geq 0}$.
Let $p_0 \in [1,2)$ and $m \in [2,\infty)$.
Assume that $(T_t)_{t \geq 0}$ satisfies generalised Gaussian estimates with parameters $p_0,m$.
Let $Y$ be a UMD lattice which is $p_Y$-convex and $q_Y$-concave for some $p_Y \in (p_0,2]$ and $q_Y \in [2,p_0')$
Assume that the convexifications $Y^{p_Y}$ and $(Y')^{q_Y'}$ are also UMD lattices.
Then
\[ \bigl\{ \bigl(\cos(\arg z)\bigr)^\alpha \exp(-zA) : \: z \in \C_+ \bigr\} \]
is $R$-bounded on $L^p(\Omega;Y)$ for $p \in (p_0,p_0'),$ where 
\[ \alpha > \alpha(p,p_Y,q_Y) d \]
from \eqref{equ-defi-alpha}.
\end{thm}

We spell out some particular cases of Theorem \ref{thm-R-bounded-semigroup}.

\begin{cor}
\label{cor-R-bounded-semigroup-Ls}
\begin{enumerate}
\item
Let the assumptions of Theorem \ref{thm-R-bounded-semigroup} be satisfied, with $Y = L^s(\Omega').$ 
Assume that $p,s \in (p_0,p_0')$.
Then $\begin{displaystyle}\bigl\{ \bigl(\cos(\arg z)\bigr)^\alpha T_z :\: z \in \C_+ \bigr\}\end{displaystyle}$ is $R$-bounded on $L^p(\Omega;L^s(\Omega'))$ for
\[\alpha > \biggl( \max \biggl(\frac{1}{p},\frac{1}{s},\frac{1}{2} \biggr) - \min \biggl( \frac{1}{p},\frac{1}{s},\frac{1}{2} \biggr) \biggr) \cdot d.\]
\item
Let $(T_t)_t$ be a self-adjoint semigroup on $L^2(\Omega)$ having (classical) Gaussian estimates.
Let $Y$ be any UMD lattice.
Then for $p \in (1,\infty)$, $\begin{displaystyle}\bigl\{ \bigl(\cos(\arg z)\bigr)^{\alpha} T_z :\: z \in \C_+ \bigr\}\end{displaystyle}$ is $R$-bounded on $L^p(\Omega;Y)$ for
$\alpha > \alpha(p,p_Y,q_Y) \cdot d \in (0,d)$.
\end{enumerate}
\end{cor}

\begin{proof}
1. It suffices to note that according to Lemma \ref{lem-Ls-convexity}, $L^s$ is $s$-convex and $s$-concave, so also $p_s$-convex and $q_s$-concave with some exponents $p_0 < p_s < \min(s,2)$ and $p_0' > q_s > \max(s,2)$ which are allowed in the assumptions of Theorem \ref{thm-R-bounded-semigroup}.
Moreover, $(L^s(\Omega'))^{p_s} = L^{\frac{s}{p_s}}(\Omega')$ and $(L^{s'}(\Omega'))^{q_s'} = L^{\frac{s'}{q_s'}}(\Omega')$ are UMD lattices since $\frac{s}{p_s},\frac{s'}{q_s'} \in (1,\infty)$. 

2.  According to Remark \ref{rem-GE-GGE}, classical Gaussian estimates are equivalent to generalised Gaussian estimates with parameter $p_0 = 1$.
Note that according to Lemma \ref{lem-UMD-lattice-convexity} any UMD lattice $Y$ is $p_Y$-convex and $q_Y$-concave for some $p_Y \in (1,2]$ and $q_Y \in [2,\infty)$ and moreover such that, $Y^{p_Y}$ and $(Y')^{q_Y'}$ are UMD.
\end{proof}

For the preparation of the proof of Theorem \ref{thm-R-bounded-semigroup}, we need two lemmas.
First, we have the following result from \cite[Proposition 2.3]{Ku08}.

\begin{lemma}
\label{lem-Ku08-Prop23}
Let $(\Omega,\dist,\mu)$ be a space of homogeneous type with a dimension $d$.
Let $1 \leq q_0 \leq q_1 \leq \infty$ and $(S(t))_{t \in \tau}$ be a family of linear operators on $L^{q_0}(\Omega) \cap L^{q_1}(\Omega).$
Recall the notation \[A(x,r,k) = B\bigl(x,(k+1)r\bigr) \backslash B(x,r).\]
Assume 
\[ \bigl\| 1_{B(x,\rho(t))} S(t) 1_{A(x,\rho(t),k)} \bigr\|_{L^{q_0}(\Omega) \to L^{q_1}(\Omega)} \leq V\bigl(x,\rho(t)\bigr)^{-(\frac{1}{q_0} - \frac{1}{q_1})} h(k) \quad(x \in \Omega, \: t \in \tau, \: k \in \N_0) \]
for some functions $\rho : \tau \to (0,\infty)$ and $h : \N_0 \to (0,\infty)$ with $h(k) \leq c (k+1)^{-\delta}$ and $\delta > \frac{d}{q_0} + \frac{1}{q_0'}$.
Then 
\[ N_{q_1,\rho(t)}\bigl(S(t)f\bigr)(x,\omega') \leq C M_{HL}^{q_0}f(x,\omega') \quad (t \in \tau, \: x \in \Omega, \: \omega' \in \Omega', \: f \in L^{q_0}(\Omega) \otimes Y).\]
\end{lemma}

\begin{proof}
Apply \cite[Proposition 2.3]{Ku08} pointwise, i.e. for fixed $\omega' \in \Omega'$.
\end{proof}

The next lemma is related to \cite[Proposition 2.4]{Ku08} (our $Y(\ell^2)$ replaces $\ell^s$ there).

\begin{lemma}
\label{lem-lower-R-bounded}
Let $Y$ be a UMD lattice which is $p_Y$-convex and $q_Y$-concave for some $p_Y \in (1,2]$ and $q_Y \in [2,\infty)$.
Assume that the convexifications $Y^{p_Y}$ and $(Y')^{q_Y'}$ are also UMD (lattices).
Let $q_0 \in [1,p_Y]$, $q_1 \in [q_Y,\infty]$ and $q \in (q_0,q_1).$
Then the family $\{ N_{q_0,r} :\: r > 0 \}$ is $R$-bounded in $L^q(\Omega;Y)$, and the family $\{ N_{q_1,r} :\: r > 0 \}$ is lower $R$-bounded in $L^q(\Omega;Y)$ (see Definition \ref{defi-lower-R-bounded}).
\end{lemma}

\begin{proof}
Note that $Y(\ell^2)^{p_Y}$ is a UMD lattice according to Lemma \ref{lem-UMD-lattice-convexity}. 
Now since $q_0 < q$ and $q_0 \leq p_Y$, we obtain from Proposition \ref{prop-MHL-q} that
\begin{align*}
\biggl\| \biggl( \sum_k |N_{q_0,r_k} f_k|^2 \biggr)^{\frac12} \biggr\|_{L^q(\Omega;Y)} & \leq \biggl\| \biggl( \sum_k |M^{q_0}_{HL} f_k|^2 \biggr)^{\frac12} \biggr\|_{L^q(\Omega;Y)} \\
& \lesssim \biggl\| \biggl( \sum_k |f_k|^2 \biggr)^{\frac12} \biggr\|_{L^q(\Omega;Y)}.
\end{align*}
In other words, the upper $R$-boundedness statement is shown.

We turn to the lower $R$-boundedness statement.
Let $\mathcal{D} = \bigcup_{k \in \Z} \mathcal{D}_k$ be a `dyadic system' and 
\[ E_{q_0,k}f(x,\omega') = \sum_{Q \in \mathcal{D}_k} \biggl( \frac{1_Q(x)}{\mu(Q)} \int_Q |f(y,\omega')|^{q_0} \, d\mu(y) \biggr)^{\frac1{q_0}} \]
be the $L^{q_0}$ version of the conditional expectation from \eqref{equ-conditional-expectation} associated with $\mathcal{D}$.
Then we claim that
\[ \biggl\| \biggl( \sum_{k} |E_{q_0,k} f_k|^2 \biggr)^{\frac12} \biggr\|_{L^q(\Omega;Y)} \lesssim \biggl\| \biggl( \sum_{k} |f_k|^2 \biggr)^{\frac12} \biggr\|_{L^q(\Omega;Y)} \]
for $q_0 < q$ and $q_0 \leq p_Y$.
Indeed, this can easily be deduced from the fact that the maximal operator associated with the $E_{q_0,k}$, which is
\[ M^{q_0}_{\mathcal{F}} (f)(x,\omega') = \sup_{\substack{Q\ni x \\ Q\in\mathcal{D}}} \biggl( \frac{1}{\mu(Q)} \int_Q |f(y,\omega')|^{q_0} \, dy
  \biggr)^{\frac{1}{q_0}} \]
is bounded on $L^q(\Omega;Y)$.
This  in turn can be shown as in the proof of Proposition \ref{prop-MHL-q} together with the fact that $M_{\mathcal{F}}$ is bounded on $L^p(\Omega;Z)$ for $1 < p < \infty$ and $Z$ a UMD lattice according to Lemma \ref{lemma-MHL}.
According to Lemma \ref{lem-dyadic-cubes-lower}, for all $r > 0$ there is some $k(r) \in \Z$ such that $x \in Q \in \mathcal{D}_{k(r)}$ implies $Q \subseteq B(x,r)$ and $V(x,r) \lesssim \mu(Q)$.
Therefore, 
\begin{align*}
N_{q_1,r_i}(f_i)(x,\omega') & = \biggl( \frac{1}{V(x,r_i)} \int_{B(x,r_i)} |f_i(y,\omega')|^{q_1} \, d\mu(y) \biggr)^{\frac{1}{q_1}} \\
& \gtrsim \biggl( \sum_{Q \in \mathcal{D}_{k(r_i)}} \frac{1}{\mu(Q)} \int_Q |f_i(y,\omega')|^{q_1} \, d\mu(y) \biggr)^{\frac{1}{q_1}} \\
& = E_{q_1,k_i}f_i(x,\omega')
\end{align*}
with $k_i = k(r_i)$.
We deduce that
\[\biggl\| \biggl( \sum_i |N_{q_1,r_i} (f_i)|^2 \biggr)^{\frac12} \biggr\|_{L^q(\Omega;Y)} 
\gtrsim \biggl\| \biggl( \sum_i |E_{q_1,k_i}(f_i)|^2 \biggr)^{\frac12} \biggr\|_{L^q(\Omega;Y)}, \]
so that it will suffice to show the lower $R$-boundedness of the family $\{ E_{q_1,k} : \: k \in \Z \}$ in $L^q(\Omega;Y)$.
To show this, we will use the already established upper $R$-boundedness together with a duality argument.
For this duality argument, we make use of the following $\sigma$-finite auxiliary measure space
\[ M = \bigsqcup_{k \in \Z} \Omega ,\: \tilde{\mu} = \bigoplus_{k \in \Z} \sum_{Q \in \mathcal{D}_k} \frac{1}{\mu(Q)} \mu|_{Q} \]
consisiting of a sequence of independent copies of $\Omega$ together with a suitable renormalised measure $\tilde{\mu}$ to fit the $E_{q_1,k}$ as we shall see in what follows.
Namely, consider the operator
\[ T : \begin{cases} L^q(\Omega;Y(\ell^2)) & \to L^q(\Omega;Y(\ell^2(L^{q_0}(M)))) \\ (f_k) & \mapsto (\tilde{f}_k) \end{cases} \]
with $\tilde{f}_k(x,\omega',j,y) = 1_{Q^k_x}(y) \delta_{k = j} f_k(y,\omega')$, where $Q^k_x$ will stand in what follows for the unique dyadic cube $Q \in \mathcal{D}_k$ containing $x$ and $(j,y)$ with $j \in \Z$ and $y \in \Omega$ is the generic variable in $M$.
Then
\begin{align*}
\| (\tilde{f}_k)_k \|_{L^q(\Omega;Y(\ell^2(L^{q_0}(M))))}^q & = \int_{\Omega} \biggl\| \biggl( \sum_k \| \tilde{f}_k(x,\omega',\cdot)\|_{L^{q_0}(M)}^2 \biggr)^{\frac12} \biggr\|_Y^q \, d\mu(x) \\
& = \int_\Omega \biggl\| \biggl( \sum_k \bigl[ \int_M  |\tilde{f}_k(x,\omega',j,y)|^{q_0} d\tilde{\mu}(j,y) \bigr]^{\frac{2}{q_0}} \biggr)^{\frac12} \biggr\|_Y^q \, d\mu(x) \\
& = \int_{\Omega} \biggl\| \biggl( \sum_k \bigl( \frac{1}{\mu(Q^k_x)} \int_{Q^k_x} |f_k(y,\omega')|^{q_0} \, d\mu(y) \bigr)^{\frac{2}{q_0}} \biggr)^{\frac12} \biggr\|_Y^q \, d\mu(x) \\
& = \int_\Omega \biggl\| \biggl( \sum_k |E_{q_0,k}(f_k)(x,\omega')|^2 \biggr)^{\frac12} \biggr\|_Y^q \, d\mu(x).
\end{align*}
Now the fact that $\{E_{q_0,k}:\:k \in \Z \}$ are upper $R$-bounded in $L^q(\Omega;Y)$ rereads as $T$ is bounded.
Hence also its adjoint
\[ T' : \begin{cases} L^{q'}(\Omega;Y'(\ell^2(L^{q_0'}(M)))) & \to L^{q'}(\Omega;Y'(\ell^2)) \\
(g_k)_k & \mapsto (\tilde{g}_k)_k
\end{cases} \] 
is bounded.
An elementary calculation gives that
\[ T'((g_k)_k)(y,\omega',j) = \delta_{j = k} \sum_{Q \in \mathcal{D}_k} \frac{1_Q(y)}{\mu(Q)} \int_Q  g_k(x,\omega',k,y) \, d\mu(x). \]
Since the assumptions for the upper $R$-boundedness statement are also satisfied wih the uplet $(q,Y,p_Y,q_0)$ replaced by $(q',Y',q_Y',q_1')$, we obtain that
\[ T' : \begin{cases} L^q(\Omega;Y(\ell^2(L^{q_1}(M)))) & \to L^q(\Omega;Y(\ell^2)) \\ (g_k)_k & \mapsto (\tilde{g}_k)_k \end{cases} \]
is bounded, with the above definition of $\tilde{g}_k$ provided $q_1 \geq q_Y$ and $q_1 > q$.
We will conclude the proof by a suitable choice of $(g_k)_k$.
Namely, let
\[ g_k(x, \omega',j,y) = 1_{Q^k_x}(y) \delta_{k = j} f_k(y,\omega') \]
for a given $(f_k)_k \in L^q(\Omega;Y(\ell^2))$. 
Then we obtain on the one hand
\begin{align*}
\|(g_k)_k\|_{L^q(\Omega;Y(\ell^2(L^{q_1}(M))))}^q & = \int_\Omega \biggl\| \biggl( \sum_k \bigl( \frac{1}{\mu(Q^k_x)} \int_{Q^k_x} |f_k(y,\omega')|^{q_1} \, d\mu(y) \bigr)^{\frac{2}{q_1}} \biggr)^{\frac12} \biggr\|_Y^q \, d\mu(x) \\
& = \bigl\| \bigl( \sum_k |E_{q_1,k}(f_k)(x,\omega')|^2 \bigr)^{\frac12} \bigr\|_{L^q(\Omega;Y)},
\end{align*} 
and on the other hand
\begin{align*}
\|T'((g_k)_k)\|_{L^q(\Omega;Y(\ell^2))}^q & = \int_\Omega \biggl\| \biggl( \sum_k \bigl( \sum_{Q \in \mathcal{D}_k} \frac{1}{\mu(Q)} \int_Q 1_Q(x) g_k(z,\omega',k,x) dz \bigr)^2 \biggr)^{\frac12} \biggr\|_Y^q \, d\mu(x) \\
& = \int_\Omega \biggl\| \biggl( \sum_k \bigl ( \frac{1}{\mu(Q^k_x)} \int_{Q^k_x} 1_{Q^k_z}(x) \, d\mu(z) f_k(x,\omega') \bigr)^2 \biggr)^{\frac12} \biggr\|_Y^q \, d\mu(x) \\
& = \bigl\| \bigl( \sum_k |1 \cdot f_k(x,\omega')|^2 \bigr)^{\frac12} \bigr\|_{L^q(\Omega;Y)}.
\end{align*}
Putting the estimate $\|T'((g_k)_k)\|_{L^q(\Omega;Y(\ell^2))} \lesssim \|(g_k)_k\|_{L^q(\Omega;Y(\ell^2(L^{q_1}(M))))}$ with the above calculations together readily gives the lower $R$-boundedness statement.
\end{proof}

With the previous lemmas in mind, we are now in a position to prove Theorem \ref{thm-R-bounded-semigroup}.

\begin{proof}[of Theorem \ref{thm-R-bounded-semigroup}]
Fix some $p \in (p_0,p_0')$.

Let \[\frac{1}{p_0} > \frac{1}{q_0} > \max \biggl(\frac{1}{p},\frac{1}{p_Y},\frac{1}{2} \biggr) \mathrm{\ and\ } \frac{1}{p_0'} < \frac{1}{q_1} < \min \biggl( \frac1p,\frac{1}{q_Y},\frac12 \biggr),\]
and let moreover \[\alpha = d \biggl(\frac{1}{q_0} - \frac{1}{q_1}\biggr) > \alpha(p,p_Y,q_Y).\]
Note that clearly, $q_0$ and $q_1$ can be chosen such that $\alpha$ is arbitrarily close to $\alpha(p,p_Y,q_Y)$.
We want to apply Lemma \ref{lem-Ku08-Prop23} to $\tau = \C_+$, $\rho : \C_+ \to (0,\infty),\: z \mapsto \bigl(\cos(\arg z)\bigr)^{-\frac{m-1}{m}} |z|^{\frac1m}$, $S(z) = \cos\bigl(\arg(z)\bigr)^{\alpha}\exp(-zA)$.
First note that the generalised Gaussian estimates
\[ \bigl\| 1_{B(x,r_t)} T_t 1_{B(y,r_t)} \bigr\|_{p_0 \to p_0'} \leq C V(x,r_t)^{-(\frac{1}{p_0} - \frac{1}{p_0'})} \exp \Biggl(-c \biggl( \frac{\dist(x,y)}{r_t} \biggr)^{\frac{m}{m-1}} \Biggr) \]
imply by \cite[Proposition 2.1 (i) (1) $u=p_0$, $v=p_0'$, $\alpha = \frac{1}{p_0} - \frac{1}{p_0'}$, $\beta = 0$ $\Longrightarrow$ (i) (1') $u=q_0 \geq p_0$, $v = q_1 \leq p_0'$, $\alpha = \frac{1}{q_0} - \frac{1}{q_1}$, $\beta = 0$]{BK05} that
\[ \bigl\| 1_{B(x,r_t)} T_t 1_{B(y,r_t)} \bigr\|_{q_0 \to q_1} \leq C' V(x,r_t)^{-(\frac{1}{q_0} - \frac{1}{q_1})} \exp \Biggl(-c' \biggl( \frac{\dist(x,y)}{r_t} \biggr)^{\frac{m}{m-1}} \Biggr) . \]
This implies by \cite[Theorem 2.1, $\omega = \frac{m}{m-1}$]{Bl07} that
\[ \bigl\| 1_{B(x,\rho(z))} \bigl(\cos(\arg z)\bigr)^\alpha T_z 1_{B(x,\rho(z))} \bigr\|_{q_0 \to q_1} \leq C'' V\bigl(x,\rho(z)\bigr)^{-(\frac{1}{q_0} - \frac{1}{q_1})} \exp \Biggl(- c'' \biggl( \frac{\dist(x,y)}{\rho(z)} \biggr)^{\frac{m}{m-1}} \Biggr) \]
for any $z \in \C_+$.
Now refer again to \cite[Proposition 2.1 (i) (1) $\Longrightarrow$ (3)]{BK05} and deduce for $z \in \C_+$, $x \in \Omega$ and $k \in \N_0$
\[ \bigl\| 1_{B(x,\rho(z))} \bigl(\cos(\arg z)\bigr)^\alpha T_z 1_{A(x,\rho(z),k)}\bigr\|_{q_0 \to q_1} \leq C''' V(x,\rho(z))^{-(\frac{1}{q_0} - \frac{1}{q_1})} \exp \bigl(-c''' k^{\frac{m}{m-1}}\bigr) . \]
Note that $h(k) = \exp \bigl(-c''' k^{\frac{m}{m-1}}\bigr) \leq c_\delta (k + 1)^{-\delta}$ for any $\delta > 0$.
Thus we can now apply Lemma \ref{lem-Ku08-Prop23} and deduce
\[ \bigl(\cos(\arg z)\bigr)^\alpha N_{q_1,\rho(z)}(T_z f)(x,\omega') \leq C M_{HL}^{q_0} f(x,\omega') \quad (z \in \C_+,\: x \in \Omega,\: \omega' \in \Omega', \: f \in L^{q_0}(\Omega) \otimes Y)) .\]
Then according to Lemma \ref{lem-lower-R-bounded} (note that $Y'(\ell^2)^{q_1'}$ is a UMD lattice), we have for $f_1,\ldots,f_n \in L^{q_0}(\Omega) \otimes Y \subseteq L^{q_0}(\Omega;Y)$ and $z_1,\ldots, z_n \in \C_+$, with notation $S(z) = \cos\bigl(\arg(z)\bigr)^{\alpha}\exp(-zA)$
\begin{align*}
\Biggl\| \biggl( \sum_i \bigl| S(z_i) f_i \bigr|^2 \biggr)^{\frac12} \Biggr\|_{L^p(\Omega;Y)} & \lesssim \Biggl\| \biggl( \sum_i \bigl| N_{q_1,\rho(z_i)}(S(z_i) f_i) \bigr|^2 \biggr)^{\frac12} \Biggr\|_{L^p(\Omega;Y)} \\
& \lesssim \Biggl\| \biggl( \sum_i \bigl|M_{HL}^{q_0} f_i\bigr|^2 \biggl)^{\frac12} \Biggr\|_{L^p(\Omega;Y)} \\
& \lesssim \Biggl\| \biggl( \sum_i |f_i|^2 \biggr)^{\frac12} \Biggr\|_{L^p(\Omega;Y)}
\end{align*}
where we have applied Lemma \ref{lem-Ku08-Prop23} and Proposition \ref{prop-MHL-q} (note that $Y(\ell^2)^{q_0}$ is a UMD lattice).
\end{proof}

Now we gather several situations, in which the operator $A$ has an $\HI$ calculus on $L^p(\Omega;Y)$.
This will be important for the H\"ormander calculus, i.e. it is one of the hypotheses in Theorem \ref{thm-Hoermander}.

\begin{thm}
\label{thm-HI-extrapolation}
Let $(\Omega,\dist,\mu)$ be a space of homogeneous type and $E$ a Banach space.
Let $A$ generate the self-adjoint semigroup $(T_t)_t$ on $L^2(\Omega)$ satisfying Gaussian estimates \eqref{equ-GE-prelims}.
Let $p_0 \in (1,\infty)$.
Assume that $A$ has an $\HI(\Sigma_\omega)$ calculus on $L^{p_0}(\Omega;E)$ for some $\omega \in (0,\pi)$.
Then for any $p \in (1,\infty)$, $A$ has an $\HI(\Sigma_\omega)$ calculus on $L^p(\Omega;E)$.
\end{thm}

\begin{proof}
Suppose that we have shown that for any $\xi \in \HI_0(\Sigma_\omega)$ and $f \in L^1(\Omega;E)$ with bounded support, we have
\begin{equation}
\label{equ-proof-thm-HI-extrapolation}
\mu\Bigl( \bigl\{ x \in \Omega :\: \|\xi(A) f(x)\|_E > \lambda \bigr\} \Bigr) \leq C \frac{1}{\lambda} \|\xi\|_{\infty,\omega} \|f\|_{L^1(\Omega;E)}.
\end{equation}
Then $\xi(A)$ is of weak type $L^1(\Omega;E) \to L^{1,\infty}(\Omega;E)$ and by assumptions, is also of strong type $L^{p_0}(\Omega;E) \to L^{p_0}(\Omega;E)$, so by the vector valued Marcinkiewicz interpolation theorem \cite[Lemma 1]{BCP}, $\xi(A)$ will be bounded on $L^p(\Omega;E)$ for any $1 < p < p_0$.
Now in the dual setting, the hypotheses of the Theorem imply that $A'$ has a bounded $\HI(\Sigma_\omega)$ calculus on $L^{p_0'}(\Omega;E')$.
Indeed, $A'$ will again be self-adjoint and the semigroup $T_t'$ generated by $A'$ will satisfy Gaussian estimates (note that $V(y,t^{\frac1m}) \leq 2^{(n+1)d} V(x,t^{\frac1m})$, and $\exp \Biggl(-\frac{c}{2} \biggl[\frac{\dist(y,x)}{t^{\frac1m}} \biggr]^{\frac{m}{m-1}} \Biggr) 2^{(n+1)d} \lesssim 1$ for $\dist(x,y) \in [2^n t^{\frac1m},2^{n+1} t^{\frac1m}]$).
Thus, applying \eqref{equ-proof-thm-HI-extrapolation} to $A'$, we obtain that $\xi(A') = \xi(A)'$ is bounded on $L^p(\Omega;E')$ for any $1 < p < p_0'$.
This shows that $\xi(A)$ is bounded on $L^q(\Omega;E)$ for $p_0 < q < \infty$, and the Theorem would follow.

It thus only remains to show \eqref{equ-proof-thm-HI-extrapolation}, which we do now, hereby following closely \cite[Proof of Theorem 3.1]{DuRo}, or its variant for $\omega \geq \frac{\pi}{2}$ from \cite[p.~104-105]{DuRo}.
Note that the additional assumption \cite[(6)]{DuRo} on the space $\Omega$ is not needed in this part.
We only indicate where Duong's and Robinson's arguments have to be modified slightly.
We use the Calder\'on-Zygmund decomposition of $f$ at height $\lambda > \frac{\|f\|_{L^1(E)}}{\mu(\Omega)}$ from \cite[Section 3.2]{CoW} in its vector-valued form from \cite[Section 2]{MoLu}.
That is, there exist functions $g,\: f_i : \Omega \to E$ and balls $B_i=B(x_i,r_i)$ such that
\begin{enumerate}
\item[($a_1$)] $f = g + h$ with $h = \sum_i f_i$,
\item[$(a_2)$] $\|g(x)\|_E \leq c \lambda$,
\item[$(a_3)$] $\supp f_i \subset B_i$ and each point of $X$ is contained in at most $M$ balls,
\item[$(a_4)$] $\|f_i\|_{L^1(\Omega;E)} \leq c \lambda \mu(B_i)$,
\item[$(a_5)$] $\sum_i \mu(B_i) \leq c \frac{1}{\lambda} \|f\|_{L^1(\Omega;E)}$.
\end{enumerate}
Note that $(a_4)$ and $(a_5)$ imply that $\|h\|_{L^1(\Omega;E)} \leq c \|f\|_{L^1(\Omega;E)}$:
Hence $\|g\|_{L^1(\Omega;E)} \leq (1 + c) \|f\|_{L^1(\Omega;E)}$.
Now decompose $h$ into the sum of two functions \[h_1 = \sum_i T_{t_i} f_i, \quad h_2 = \sum_i (\Id - T_{t_i}) f_i,\] where $t_i = r_i^m$, $m \geq 2$ being the parameter in the Gaussisan estimates.
At first, we estimate
\begin{multline}
\label{equ-2-proof-thm-HI-extrapolation}
\mu\Bigl( \bigl\{ x \in \Omega:\: \|\xi(A)f(x)\|_E > \lambda \bigr\} \Bigr) \\ \leq \mu\Bigl( \Bigl\{ x \in \Omega:\: \|\xi(A)g(x)\|_E > \frac{\lambda}{3} \Bigr\} \Bigr) + \sum_{i = 1}^2 \mu \Bigl( \Bigl\{ x \in \Omega: \: \| \xi(A)h_i(x)\|_E > \frac{\lambda}{3} \Bigr\} \Bigr) .
\end{multline}
For the ``good'' part $g$, we have
\begin{align*}
\mu\Bigl( \bigl\{ x \in \Omega :\: \| \xi(A)g(x) \|_E > \lambda \bigr\} \Bigr) & \leq \lambda^{-p_0} \int_\Omega \|\xi(A)g(x)\|_E^{p_0} \,dx \\
& \lesssim \lambda^{-p_0} \int_\Omega \|g(x)\|_E^{p_0} \,dx \\
& \lesssim \lambda^{-1} \int_\Omega \|g(x)\|_E \,dx \lesssim \lambda^{-1} \|f\|_{L^1(\Omega;E)}.
\end{align*}
Here we have used $(a_2)$.
Next consider the $h_1$-term in \eqref{equ-2-proof-thm-HI-extrapolation}.
We have \[\mu \Bigl( \bigl\{x \in \Omega:\: \|\xi(A)h_1(x)\|_E > \lambda \bigr\} \Bigr) \lesssim \lambda^{-p_0} \Bigl\| \sum_i T_{t_i} f_i \Bigr\|_{L^{p_0}(E)}^{p_0}.\]
Now arguing as in \cite[Proof of Theorem 3.1]{DuRo}, we obtain by the boundedness of the scalar Hardy-Littlewood maximal operator $M_{HL} : L^{p_0'}(\Omega) \to L^{p_0'}(\Omega)$ that $\|\sum_i T_{t_i} f_i\|_{L^{p_0}(E)} \lesssim \lambda \| \sum_i 1_{B_i} \|_{p_0}$.
Then using $(a_3)$ we obtain \[\mu\Bigl(\bigl\{ x \in \Omega :\: \|\xi(A)h_1(x)\|_E > \lambda \bigr\} \Bigr) \lesssim \lambda^{-1} \|f\|_{L^1(\Omega;E)}.\]
For the $h_2$-term in \eqref{equ-2-proof-thm-HI-extrapolation}, one does not need the $\HI(\Sigma_\omega)$ calculus on $L^{p_0}(E)$ any more, but the Gaussian estimates come into play.
Since the estimates of $\xi(A)h_2 = \sum_i\xi_i(A) f_i$ used in \cite[Proof of Theorem 3.1]{DuRo} are of the form $|\xi_i(A)f_i(x)| \leq \int_\Omega |k(x,y)| \: |f_i(y)| \,dy$, with $k(x,y)$ being the integral kernel of $\xi_i(A)$, and then estimating $|k(x,y)|$ further, the same arguments apply literally to our vector valued case, replacing absolute values around $f_i$ by $E$-norms.
One finally obtains, taking into account $(a_5)$, that \[\mu\Bigl( \bigl\{ x \in \Omega: \: \|\xi(A)h_2(x)\|_E > \lambda \bigr\} \Bigr) \lesssim \lambda^{-1} \|f\|_{L^1(\Omega;E)}.\]
This concludes the proof of \eqref{equ-proof-thm-HI-extrapolation}, and thus of the theorem.
\end{proof}

For generalised Gaussian estimates, we have the following result on $\HI$ calculus.

\begin{thm}
\label{thm-HI-Kunstmann-Ullmann}
Let $(\Omega,\dist,\mu)$ be a space of homogeneous type and $A$ generate a self-adjoint semigroup $(T_t)_t$ on $L^2(\Omega)$ satisfying generalised Gaussian estimates \eqref{equ-GGE} with parameters $p_0 \in [1,2)$ and $m \in [2,\infty)$.
Let $Y = L^s(\Omega')$ with $s \in (p_0,p_0')$.
Then $A$ has a bounded $\HI(\Sigma_\omega)$ calculus on $L^p(\Omega;L^s(\Omega'))$ for any $p \in (p_0,p_0')$ and $\omega \in (0,\pi)$.
\end{thm}

\begin{proof}
According to \cite[Proposition 2.1 (i) (1) $\Longrightarrow$ (3) with $\alpha = \frac{1}{p_0} - \frac{1}{p_0'}$, $\beta = 0$, $u = p_0$, $v = p_0'$]{BK05} and the dual estimate due to self-adjointness of the semigroup, the generalised Gaussian estimates \eqref{equ-GGE} imply the hypotheses of \cite[Theorem 2.3]{KuUl}.
Then \cite[Theorem 2.3]{KuUl} implies that $A$ has an $R_s$-bounded $\HI(\Sigma_\omega)$ calculus on $L^p(\Omega)$ for $s,p \in (p_0,p_0')$ and $\omega > 0$, $R_s$ boundedness being defined in that article.
By \cite[Theorem 2.1]{KuUl}, $A$ has then a bounded $\HI(\Sigma_\omega)$ calculus on $L^p(\Omega;\ell^s)$.
Now let $f = \sum_k f_k \otimes 1_{A_k} \in L^p(\Omega) \otimes L^s(\Omega')$ with $1_{A_k}$ indicator functions of pairwise disjoint measurable subsets $A_k$ of $\Omega'$ of finite positive measure, and $\xi \in \HI(\Sigma_\omega)$.
Note that clearly, those $f$ are dense in $L^p(\Omega;L^s(\Omega'))$.
We have
\begin{align*}
\| \xi(A) f\|_{L^p(\Omega;L^s(\Omega'))} & = \biggl\| \sum_k \xi(A)f_k \otimes 1_{A_k} \biggr\|_{L^p(\Omega;L^s(\Omega'))} \\
& = \Biggl( \int_\Omega \biggl\| \sum_k \xi(A)f_k(x) 1_{A_k}(\omega') \biggr\|_{L^s(\Omega')}^{p} \,dx \Biggr)^{\frac1p} \\
& = \Biggl( \int_\Omega \biggl( \sum_k \Bigl| \xi(A)f_k(x) \mu(A_k)^{\frac1s} \Bigr|^s \biggr)^{\frac{p}{s}} \,dx \Biggr)^{\frac1p} \\
& \lesssim \Biggl( \int_\Omega \biggl| \sum_k f_k(x) \mu(A_k)^{\frac1s} \biggr|^{\frac{p}{s}} \,dx \Biggr)^{\frac1p} \\
& = \biggl\| \sum_k f_k \otimes 1_{A_k} \biggr\|_{L^p(\Omega;L^s(\Omega'))} = \|f\|_{L^p(\Omega;L^s(\Omega'))}.
\end{align*}
In other words, the case $Y = L^s(\Omega')$ can be reduced to the case $Y = \ell^s$, since $L^s(\Omega')$ is representable in $\ell^s$.
\end{proof}

The following sufficient conditions for the $\HI$ calculus on $L^p(\Omega;Y)$ are essentially known in the literature.

\begin{prop}
\label{prop-HI-diffusion}
Let $(\Omega,\mu)$ be a $\sigma$-finite measure space.
\begin{enumerate}
\item 
Let $(T_t)_t$ be a semigroup acting on $L^p(\Omega)$ for some fixed $p \in (1,\infty),$ such that the $T_t$ are regular contractive, that is, there exist $S_t$ positive and contractive operators on $L^p(\Omega)$ such that $|T_t f| \leq S_t |f|$ for all $t > 0$.
Assume that $p \neq 2$ or that $T_t$ are themselves positive.
Then the generator $A$ of $(T_t)_t$ has an $\HI(\Sigma_\omega)$ calculus on $L^p(\Omega;Y)$ for any $\omega \in (\frac{\pi}{2},\pi)$ and any UMD space $Y$.
\item
Let $(T_t)_t$ be a semigroup which is contractive on $L^p(\Omega)$ for all $p \in [1,\infty]$ (strong continuity only for finite $p$).
Then the generator $A$ of $(T_t)_t$ has an $\HI(\Sigma_\omega)$ calculus on $L^p(\Omega;Y)$ for any $p \in (1,\infty)$, $\omega \in (\frac{\pi}{2},\pi)$ and any UMD space $Y$.
\end{enumerate}
\end{prop}

\begin{proof}
1. According to \cite[Theorem 4.2.1 \& p.~45]{Fen}, there exists a one parameter group $(U_t)_{t \in \R}$ of isometries acting on $L^p(\tilde{\Omega})$ for some other measure space $\tilde{\Omega}$ and positive contractions $J : L^p(\Omega) \to L^p(\tilde{\Omega})$, $P : L^p(\tilde{\Omega}) \to L^p(\Omega)$ such that $T_t f = P U_t J f$ for any $t > 0$ and $f \in L^p(\Omega)$.
Now we have for $\xi \in \HI_0(\Sigma_\omega)$ that $\xi(A) = P \xi(B) J$ with $B$ the generator of $(U_t)_t$.
Since $Y$ is UMD, according to \cite[Theorem 5]{HiPr}, $\xi(B)$ admits a bounded extension to $L^p(\tilde\Omega;Y)$ with norm $\lesssim \|\xi\|_{\infty,\omega}$.
Moreover, since $P$ and $J$ are positive, they admit bounded extensions to the $Y$ valued $L^p$ spaces, too.
Thus $\xi(A)$ also admits a bounded extension to $L^p(\Omega;Y)$ with norm $\lesssim \| \xi \|_{\infty,\omega}$.

2. It is well-known that such a semigroup satisfies the assumptions of 1. if $p \neq 2$ \cite[Theorem 2.2.1]{Tagg}.
Thus the result then follows from 1.
If $p = 2$, then we interpolate between $L^{2+\epsilon}(\Omega;Y)$ and $L^{2-\epsilon}(\Omega;Y)$.
\end{proof} 

We record the following corollary.

\begin{cor}
\label{cor-HI-positive}
Let $(\Omega,dist,\mu)$ be a space of homogeneous type and $Y$ a UMD space.
Suppose that the self-adjoint semigroup $(T_t)_t$ on $L^2(\Omega)$ satisfies the Gaussian estimates \eqref{equ-GE-prelims} and that $(T_t)_t$ is moreover (lattice) positive, i.e. $p_t(x,y) \geq 0$ for any $t > 0$ and $x,y \in \Omega$, where $p_t(x,y)$ is the integral kernel as in \eqref{equ-GE-prelims}.
Then for any $1 < p < \infty$, the generator $A$ has an $\HI(\Sigma_\omega)$ calculus on $L^p(\Omega;Y)$ for any $\omega \in (\frac{\pi}{2},\pi)$.
\end{cor}

\begin{proof}
Observe that the Gaussian bound \eqref{equ-GE-prelims} implies that $\sup_{t \geq 0} \|T_t\|_{2 \to 2} < \infty$ (see also Remark \ref{rem-GGE-spectrum}).
Thus, the spectrum of $A$ is contained in $[0,\infty)$, so that in fact, $\|T_t\|_{2 \to 2} \leq 1$.
Now apply first Proposition \ref{prop-HI-diffusion} 1. for $p = 2$ and then extrapolate via Theorem \ref{thm-HI-extrapolation} to the general case $1 < p < \infty$.
\end{proof}

Now we state the main theorem of this section.
For the existence of the $\HI$ calculus assumption needed below, we refer to Theorems \ref{thm-HI-extrapolation}, \ref{thm-HI-Kunstmann-Ullmann}, Proposition \ref{prop-HI-diffusion} and Corollary \ref{cor-HI-positive}.

\begin{thm}
\label{thm-Hoermander}
Let $(\Omega,\dist,\mu)$ be a space of homogeneous type with a dimension $d$.
Let $A$ be a self-adjoint operator on $L^2(\Omega)$ generating the semigroup $(T_t)_{t \geq 0}$.
Let $p_0 \in [1,2)$ and $m \in [2,\infty)$.
Assume that $(T_t)_{t \geq 0}$ satisfies generalised Gaussian estimates with parameters $p_0,m$.
Let $Y$ be a UMD lattice which is $p_Y$-convex and $q_Y$-concave for some $p_Y \in (p_0,2]$ and $q_Y \in [2,p_0')$.
Assume that the convexifications $Y^{p_Y}$ and $(Y')^{q_Y'}$ are also UMD lattices.
Finally, assume that $A$ has a bounded $\HI(\Sigma_\omega)$ calculus on $L^p(\Omega;Y)$ for some fixed $p \in (p_0,p_0')$ and $\omega \in (0,\pi)$.

Then $A$ has a H\"ormander $\Hor^\beta_2$ calculus on $L^p(\Omega;Y)$ with
\[ \beta > \alpha(p,p_Y,q_Y) \cdot d + \frac12\]
and $\alpha$ from \eqref{equ-defi-alpha}.
\end{thm}

\begin{proof}
Note that according to Lemma \ref{lem-property-alpha}, $L^p(\Omega;Y)$ has Pisier's property $(\alpha)$.
Now combine Theorem \ref{thm-R-bounded-semigroup} with Theorem \ref{thm-KrW3}.
\end{proof}

Theorem \ref{thm-KrW3}, used in the proof of Theorem \ref{thm-Hoermander} admits an enhancement, and this carries over to Theorem \ref{thm-Hoermander}.
This is the content of the next theorem.

\begin{thm}
\label{thm-R-Hor}
Let $A$ have a $\Hor^\beta_2$ calculus as a consequence of any of theorems in this article.
Then in fact, the operators $f(A)$ are not only bounded for $f \in \Hor^\beta_2$, but also $R$-bounded for a whole family of functions $f$, under the condition that $\|f\|_{\Hor^\beta_2}$ remains bounded.
In other words, there exists  a constant $C < \infty$ such that for any $x_1,\ldots,x_n \in L^p(\Omega;Y)$ and $f_1,\ldots,f_n \in \Hor^\beta_2$, we have a square function estimate
\[ \Biggl\| \biggl( \sum_{i = 1}^n \bigl|f_i(A)x_i\bigr|^2 \biggr)^{\frac12} \Biggr\|_{L^p(\Omega;Y)} \leq C \sup_{i = 1,\ldots,n} \|f_i\|_{\Hor^\beta_2} \Biggl\| \biggl( \sum_{i = 1}^n |x_i|^2 \biggr)^{\frac12} \Biggr\|_{L^p(\Omega;Y)} . \]
\end{thm}

\begin{proof}
Theorem \ref{thm-KrW3} in the form from \cite[Theorem 7.1]{KrW3} immediately gives Theorem \ref{thm-R-Hor}, once one notices that $L^p(\Omega;Y)$ has property $(\alpha)$, according to Lemma \ref{lem-property-alpha}.
\end{proof}

\begin{remark}
\label{rem-comparison}
\begin{enumerate}
\item
Spectral multiplier theorems under generalised Gaussian estimates have been obtained in the last 5 years by different methods, but only in the scalar case $Y = \C$.
We note that in this case, $p_Y = q_Y = 2$ and our H\"ormander functional calculus exponent becomes \[\beta > \alpha(p,p_Y,q_Y) \cdot d + \frac12 = \biggl(\max\biggl(\frac1p,\frac12\biggr) - \min\biggl(\frac1p,\frac12\biggr)\biggr)d + \frac12 = \Bigl|\frac1p - \frac12\Bigr| d + \frac12.\]
Let us compare this result with those scalar valued H\"ormander type spectral multiplier theorems obtained in the literature, sometimes under stronger hypotheses.
In the literature, the Definition \ref{defi-Hoermander-class} of $\Hor^\beta_q$ is extended for values $q \neq 2$ in an obvious manner.
We denote $(GGE_{p_0,m})$ for $p_0 \in [1,2)$ and $m \geq 2$ our generalised Gaussian estimate hypothesis, and refer to the sources below for the definition of other hypotheses.
In all cases, the semigroup is moreover assumed to be self-adjoint, acting on $L^2(\Omega)$ with $\Omega$ a space of homogeneous type.
Finally, in the last two sources, there is an autoimprovement of the calculus by self-adjointness of $T_t$ possible.

\noindent
\begin{tabular}{| l | l | l |}
\hline			
Resource & Hypotheses & $\Hor^\beta_q$ calculus on $L^p(\Omega;\C)$ with $p_0 < p < p_0'$ \\
\hline
This article, Theorem \ref{thm-Hoermander} & $(GGE_{p_0,m})$ & $\beta > |\frac1p - \frac12| d + \frac12$ , $q = 2$  \\
\hline
\cite[Theorem 1.1]{Bl} & $(GGE_{p_0,m})$ & $\beta > \frac{d}{2} + \frac12$, $q = 2$ \\
\hline  
\cite[Theorem 5.4 a)]{KuUhl} & $(GGE_{p_0,m})$ & $\beta > | \frac1p - \frac12 | (d+1)$, $\frac1q < |\frac1p - \frac12|$ \\
\hline
\cite[Theorem 5.4b)]{KuUhl} & $(GGE_{p_0,m})$ & $\beta > | \frac1p - \frac12 | d$, $q = \infty$ \\
\hline
\cite[Theorem 4.1]{COSY} & $(FS) + (ST^q_{p_0,2})$ & $\beta > \max(d(\frac1{p_0} - \frac12),\frac1q)$ \\
\hline
\cite[Theorem 5.1]{SYY} & $(DG_m) + (ST^q_{p_0,2,m})$ & $\beta > \max(d(\frac1{p_0} - \frac12), \frac1q)$ \\
\hline
\end{tabular}

\item In the case that $\Omega = \R^d$ and $(T_t)_t$ satisfying classical Gaussian estimates \eqref{equ-GE}, a combination of \cite{ALV} and \cite{GoY} also yields UMD lattice valued spectral multipliers.
Indeed, in \cite{ALV} it is shown that if $m(A)$ satisfies weighted estimates $L^p(\R^d,w) \to L^p(\R^d,w)$ for any weight $w \in A_{p/r_0}$ in the classical Muckenhoupt class, and any $r_0 < p < \infty$, then it extends boundedly to $L^p(\R^d;Y) \to L^p(\R^d;Y)$ for $r_0 < p < \infty$ (in fact, even to $L^p(\R^d,w;Y) \to L^p(\R^d,w;Y)$ for such weights $w$), where $r_0 = p_Y$ is the convexity exponent of $Y$.
On the other hand, \cite{GoY} establishes such scalar weighted estimates $m(A) : L^p(\R^d,w) \to L^p(\R^d,w)$.
Going into the parameter calculations in \cite{GoY,ALV}, one obtains that $A$ has a bounded $\Hor^\beta_\infty$ calculus on $L^p(\R^d;Y)$ for $\beta > \frac{d}{p_Y}$ and $p_Y < p < \infty$ and for $\beta > \frac{d}{q_Y'}$ and $1 < p < q_Y$.
This result and ours from Theorem \ref{thm-Hoermander} are incomparable, since this $\Hor^\beta_\infty$ class and our $\Hor^\beta_2$ class are not contained in each other, also due to the fact that we take into account the concavity exponent $q_Y$ in addition to the convexity exponent $p_Y$.
Moreover, we also obtain square function estimates in Theorem \ref{thm-R-Hor}.
On the other hand, \cite{GoY,ALV} obtain weighted UMD lattice valued estimates.
\end{enumerate}
\end{remark}

We spell out some particular cases of Theorem \ref{thm-Hoermander}.

\begin{cor}
\label{cor-Hoermander}
Let $(\Omega,\dist,\mu)$ be a space of homogeneous type with a dimension $d$.
Let $A$ be a self-adjoint operator on $L^2(\Omega)$ generating the semigroup $(T_t)_{t \geq 0}$.
\begin{enumerate}
\item
Let $p_0 \in [1,2)$ and $m \in [2,\infty)$.
Assume that $(T_t)_{t \geq 0}$ satisfies generalised Gaussian estimates with parameters $p_0,m$.
Let $Y = L^s(\Omega')$.
Assume that $p,s \in (p_0,p_0')$.
Then $A$ has a H\"ormander $\Hor^\beta_2$ calculus on $L^p(\Omega;Y)$ with
\[ \beta > \biggl( \max \biggl(\frac{1}{p},\frac{1}{s},\frac{1}{2} \biggr) - \min \biggl( \frac{1}{p},\frac{1}{s},\frac{1}{2} \biggr) \biggr) \cdot d + \frac12. \]
\item
Assume that $(T_t)_t$ satisfies (classical) Gaussian estimates.
Let $Y$ be any UMD lattice.
Let $p \in (1,\infty)$.
Assume that $A$ has a bounded $\HI(\Sigma_\omega)$ calculus on $L^p(\Omega;Y)$ for some $\omega \in (0,\pi)$.
Then $A$ has a H\"ormander $\Hor^\beta_2$ calculus on $L^p(\Omega;Y)$ for
\[\beta > \alpha(p,p_Y,q_Y) \cdot d + \frac12.\]
\end{enumerate}
\end{cor}

\begin{proof}
1. Take into account Corollary \ref{cor-R-bounded-semigroup-Ls} and Theorems \ref{thm-KrW3} and \ref{thm-HI-Kunstmann-Ullmann}.

2. Take into account Corollary \ref{cor-R-bounded-semigroup-Ls} and Theorem \ref{thm-KrW3}.
\end{proof}

If we plug in $f_{\delta,u}(t) = (1 - \frac{t}{u})_+^\delta$ into the H\"ormander functional calculus, we obtain the following result.

\begin{cor}
\label{cor-Bochner-Riesz}
Let $(\Omega,\dist,\mu)$ be a space of homogeneous type with a dimension $d$.
Let $A$ be a self-adjoint operator on $L^2(\Omega)$ generating the semigroup $(T_t)_{t \geq 0}$.
Let $p_0 \in [1,2)$ and $m \in [2,\infty)$.
Assume that $(T_t)_{t \geq 0}$ satisfies generalised Gaussian estimates with parameters $p_0,m$.
Let $Y$ be a UMD lattice which is $p_Y$-convex and $q_Y$-concave for some $p_Y \in (p_0,2]$ and $q_Y \in [2,p_0')$.
Assume that the convexifications $Y^{p_Y}$ and $(Y')^{q_Y'}$ are also UMD lattices.
Assume moreover that $A$ has a bounded $\HI(\Sigma_\omega)$ calculus on $L^p(\Omega;Y)$ for some fixed $p \in (p_0,p_0')$.
Then the Bochner-Riesz means associated with $A$ satisfy
\[ \sup_{u > 0} \biggl\|\biggl(1 - \frac{1}{u} A\biggr)_+^\delta\biggr\|_{L^p(\Omega;Y) \to L^p(\Omega;Y)} < \infty ,\]
provided that $\delta > \alpha(p,p_Y,q_Y) \cdot d$.
Moreover, for these $\delta$, we have
\[ \biggl\| \biggl( \sum_k \bigl| \bigl(1 - \frac{1}{u_k} A\bigr)_+^\delta f_k \bigr|^2 \biggr)^{\frac12} \biggr\|_{L^p(\Omega;Y)} \leq C  \biggl\| \biggl( \sum_k | f_k |^2 \biggr)^{\frac12} \biggr\|_{L^p(\Omega;Y)} \]
for any $u_k > 0$ and $f_k \in L^p(\Omega;Y)$.
\end{cor}

\begin{proof}
For the first part, it suffices to apply Theorem \ref{thm-Hoermander} and to note the H\"ormander norm estimate
\[ \sup_{u > 0} \|f_{\delta,u}\|_{\Hor^\beta_2} < \infty \]
provided that $\delta > \beta - \frac12$ \cite[p.~11 in arxiv version]{COSY}.
Then for the second part, apply Theorem \ref{thm-R-Hor}.
\end{proof}

Another application of Theorem \ref{thm-Hoermander} is the following spectral decomposition of Paley-Littlewood type.
We refer e.g. to \cite{KrW2} for applications of this decomposition to the description of complex and real interpolation spaces associated with an abstract operator $A$.
To this end, we let $(\phi_n)_{n \in \Z}$ be a dyadic partition of unity in the sense of Definition \ref{defi-dyadic-partition}.
Further let $\psi_n = \phi_n$ for $n \geq 1$ and $\psi_0 = \sum_{n = -\infty}^0 \phi_n,$ so that $\sum_{n \in \Z} \phi_n(t) = \sum_{n = 0}^\infty \psi_n(t) = 1$ for all $t > 0.$

\begin{cor}
\label{cor-Littlewood-Paley}
Let $(\Omega,\dist,\mu)$ be a space of homogeneous type with a dimension $d$.
Let $A$ be a self-adjoint operator on $L^2(\Omega)$ generating the semigroup $(T_t)_{t \geq 0}$.
Let $p_0 \in [1,2)$ and $m \in [2,\infty)$.
Assume that $(T_t)_{t \geq 0}$ satisfies generalised Gaussian estimates with parameters $p_0,m$.
Assume moreover that $A$ has a bounded $\HI(\Sigma_\omega)$ calculus on $L^p(\Omega;Y)$ for some fixed $p \in (p_0,p_0')$.
Let $Y = Y(\Omega')$ be a UMD Banach lattice.

Then, for any $f \in L^p(\Omega;Y)$, we have the norm description
\[ \| f \|_{L^p(\Omega;Y)} \cong \Biggl\| \biggl( \sum_{ n \in \Z } \bigl| \phi_n(A) f \bigr|^2 \biggr)^{\frac12} \Biggr\|_{L^p(\Omega;Y)} \cong \Biggl\| \biggl( \sum_{ n = 0 }^\infty \bigl | \psi_n(A) f \bigr |^2 \biggr)^{\frac12} \Biggr\|_{L^p(\Omega;Y)} . \]
\end{cor}

\begin{proof}
Once a H\"ormander calculus of $A$ on $L^p(\Omega;Y)$ is guaranteed by Theorem \ref{thm-Hoermander}, the corollary follows from \cite[Theorem 4.1]{KrW2} resp. \eqref{equ-Rademacher-square}, to decompose the norm in Rademacher sums resp. square sums.
\end{proof}

\begin{remark}
\label{rem-interpolation}
We note that if $Y(\Omega') = [Z(\Omega'),L^2(\Omega')]_\theta$ is a complex interpolation space with $Z$ a further UMD lattice and $\theta \in (0,1)$, then one can apply complex interpolation to improve the derivation exponent in the H\"ormander calculus of Theorem \ref{thm-Hoermander}.
Note however that one passes from an exponent which is maybe not optimal to another one again not optimal.
The interpolation procedure goes like this.
Introduce in the H\"ormander classes a second parameter $q \in [1,\infty]$ and define $\Hor^\beta_q$ by the norm
\[ \|f\|_{\Hor^\beta_q} = |f(0)| + \sup_{R > 0} \|\phi f(R\cdot)\|_{W^\beta_q(\R)} \]
in a similar manner to Definition \ref{defi-Hoermander-class}.
Then $A$ has a $\Hor^\beta_q$ calculus on $L^2(\Omega;L^2(\Omega'))$ provided that $\Hor^\beta_q \hookrightarrow C_b([0,\infty))$
due to the self-adjoint spectral calculus, which is the case for $\beta > \frac1q$.
We assume for simplicity that $T_t$ satisfies classical Gaussian estimates.
Then Theorem \ref{thm-Hoermander} gives a $\Hor^\beta_2$ calculus on $L^p(\Omega;Z)$ for $p$ close to $1$ or close to $\infty$ with a certain $\beta$.
According to \cite[Proposition 4.83]{KrPhD}, one can interpolate the mappings
\begin{align*}
\Hor^\beta_2 & \to B(L^p(\Omega;Z)),\: f \mapsto f(A) \\
\intertext{and}
\Hor^\epsilon_\infty & \to B(L^2(\Omega;L^2(\Omega'))),\: f \mapsto f(A)
\end{align*}
to obtain a calculus $\Hor^{\beta_\theta}_{q_\theta} \to B(L^{p_\theta}(\Omega;[Z,L^2]_\theta)) = B(L^{p_\theta}(\Omega;Y))$.
Going through the calculation,
one gets in case $\frac{1}{p_Y} = 1 - \frac{\theta}{2}$ and $\frac{1}{q_Y} = \frac{\theta}{2}$ that $\beta_\theta > 2d|\frac1p - \frac12| + |\frac1p - \frac12|$ for $p \in (1, p_Y)$ or $p \in (q_Y,\infty)$ and a certain $q_\theta \in (2,\infty)$.
For $p$ close to $p_Y$, the differentiation index $\beta_\theta$ is close to $\frac12 - |\frac1{p_Y} - \frac12|$ better than what gives Theorem \ref{thm-Hoermander}.
\end{remark}

\begin{remark}
We point out a different strategy to show a weaker Mihlin functional calculus on $L^p(\Omega;Y)$ in a slightly different setting.
Namely, assume the conditions at the beginning of \cite[Section 4]{Kempp2}.
That is, $(\Omega,\dist,\mu)$ is a complete space of homogeneous type having in addition the cone property, $A$ is self-adoint on $L^2(\Omega)$ satisfying the more general Davies-Gaffney estimates, which correspond to generalised Gaussian estimates as in \eqref{equ-GGE} with parameters $p_0 = 2, \: m = 2$, and finally, $Y$ is a \textbf{UMD space} (not necessarily a lattice).

Then \cite[Theorem 12]{Kempp2} yields a $\HI(\Sigma_\omega)$ calculus for any angle $\omega > 0$ on the vector-valued Hardy space $H^p(\Omega;Y)$, $p \in (1,\infty)$.
If $A = - \Delta + V$ is a Schr\"odinger operator on $\Omega = \R^d$ with a positive potential $V \geq 0$ satisfying the following reverse H\"older condition for some $s > \frac{d}{2}$ and $C < \infty$
\[ \biggl( \int_B V^s(x) \,dx \biggr)^{\frac1s} \leq C \int_B V(x) \, dx \]
for any ball $B \subseteq \R^d$, \cite[(3)]{BCFR}, then we have \cite[Remark p. 18]{Kempp2} that $H^p(\Omega;Y) = L^p(\Omega;Y)$.
Now it is known that such a calculus $\Phi_\omega: \HI(\Sigma_\omega) \to B(L^p(\Omega;Y))$ for \textit{any} angle is related to Mihlin calculus \cite[Theorem 4.10]{CDMY}.
An inspection of the proof of \cite[Theorem 12]{Kempp2} shows that $\|\Phi_\omega\| \leq C \omega^{-d-1}$, where $d \in \N$ is a doubling dimension.
This implies by \cite[Theorem 4.10]{CDMY} that
\[ \|\xi(A)\|_{L^p(\Omega;Y) \to L^p(\Omega;Y)} \lesssim \max_{k = 0 , \ldots, d+1} \sup_{t > 0} t^k |\xi^{(k)}(t)| =: \|\xi\|_{M^{d+1}}. \]
Note that $\|\xi\|_{\Hor^\beta_2} \leq \|\xi\|_{\Hor^{d+1}_2} \lesssim \|\xi\|_{M^{d+1}}$, with $\beta = \alpha(p,p_Y,q_Y) \cdot (d + 1) + \epsilon < d+1$ from Theorem \ref{thm-Hoermander}, so that the latter Theorem yields a stronger result, when applicable.
\end{remark}

\subsection{A simpler alternate approach for the pure Laplacian, (classical) Gaussian estimates and dispersive estimates}
\label{subsec-simpler-approach}

If one does not strive for the optimal parameters $\alpha$ and $\beta$ in Theorems \ref{thm-R-bounded-semigroup} and \ref{thm-Hoermander}, then in the case of classical Gaussian estimates, there is a simpler and more direct approach which does not need the machinery of Blunck's and Kunstmann's work, but rather an extrapolation of Gaussian estimates from real to complex time from \cite[Proposition 4.1]{CaCoOu}.
One finds $\beta > d + 1$, see Corollary \ref{cor-GE-Hoermander} 1.
Moreover, in the case of $A$ being the pure Laplacian on $L^p(\R^d)$, this approach yields a parameter $\beta > \frac{d+1}{2}$ for any UMD lattice and any $1 < p < \infty$ independent of convexity and concavity of the UMD lattice $Y$.
Also if one knows dispersive estimates for $\exp(itA)$ and the volume of $\Omega$ is polynomially bounded, then the estimate in \cite[Proposition 4.1]{CaCoOu} can be improved, and again the parameter becomes $\beta > \frac{d+1}{2}$, see Corollary \ref{cor-GE-Hoermander} 2.
This parameter, universal in the class of UMD lattices, can then be strengthened, by self-adjoint calculus and a complex interpolation argument for 
\[ L^p(\Omega;Y) = L^p(\Omega;[Z,L^2]_\theta) = [L^{p_1}(\Omega;Z),L^2(\Omega;L^2)]_\theta\]
with $p_1$ close to $1$ or $\infty$.
It will give the condition $\beta > \widetilde{\alpha}(p,p_Y,q_Y) \cdot d + \frac12$ with
\[ \widetilde{\alpha}\bigl(p,p_Y,q_Y\bigr) = \max\biggl(\bigl|\frac1p - \frac12\bigr|, \bigl|\frac1{p_Y}-\frac12\bigr|,\bigl|\frac1{q_Y} - \frac12\bigr|\biggr) \leq \alpha(p,p_Y,q_Y), \]
see Remark \ref{rem-interpolation-dispersive}.
We start with the case of $A = -\Delta$ on $L^p(\R^d)$.
Then it is well known that $T_z = \exp(z \Delta)$ has the Gaussian integral kernel \[p_z(x,y) = \frac{1}{\sqrt{4 \pi z}^d} \exp \biggl(- \frac{|x-y|^2}{4z} \biggr)\] for $z \in \C_+$.

\begin{prop}
\label{prop-pure-laplacian}
For $\theta \in (- \frac{\pi}{2}, \frac{\pi}{2}),$ let  $M_\theta$ denote the maximal operator 
\[M_\theta(f)(x,\omega') = \sup_{t > 0} \bigl|T_{te^{i\theta}}(f(\cdot,\omega'))(x)\bigr| \]
for $f \in L^p(\R^d) \otimes Y$.
If $p \in (1,\infty)$ and $Y$ is a UMD lattice, then $M_\theta$ extends to a bounded operator on $L^p(\R^d;Y)$ with 
\[ \|M_\theta f\|_{L^p(\R^d;Y)} \leq C_{Y,p} \bigl(\cos(\theta)\bigr)^{-\frac{d}{2}} \|f\|_{L^p(\R^d;Y)} . \]
\end{prop}

\begin{proof}
We estimate
\begin{align*}
M_\theta f(x,\omega') & = \sup_{t > 0} \bigl| T_{te^{i\theta}} (f(\cdot,\omega'))(x)\bigr| \\
 & \leq \sup_{t > 0} \int_{\R^d} \biggl| \frac{1}{\sqrt{4 \pi e^{i\theta}t}^d} \exp \biggl(-\frac{|x-y|^2}{4 e^{i\theta} t} \biggr) \biggr| |f(y,\omega')| \,dy \\
& \leq \frac{1}{(4\pi)^{\frac{d}{2}}} \sup_{t > 0} \int_{\R^d} \frac{1}{t^{\frac{d}{2}}} \exp \biggl( - \frac{|x-y|^2}{4t} \cos (\theta) \biggr) |f(y,\omega')| \,dy \\
& = \frac{1}{(4\pi)^{\frac{d}{2}}} \sup_{t > 0} \int_{\R^d} \frac{1}{\bigl(\cos(\theta) t\bigr)^{\frac{d}{2}}} \exp\biggl( - \frac{|x-y|^2}{4t} \biggr) |f(y,\omega')| \,dy \\
& = C_d \bigl(\cos(\theta)\bigr)^{-\frac{d}{2}} M_0 f(x,\omega'),
\end{align*}
where we have performed the simple substitution $t \mapsto \cos(\theta) t$ in the second to last step.
It is well-known that $M_0 f(x,\omega') \leq C_d M_{HL}f(x,\omega')$ pointwise.
Thus, we deduce from Theorem \ref{thm-MHL}
\begin{align*} \|M_\theta f\|_{L^p(\R^d;Y)} & \lesssim \bigl(\cos(\theta)\bigr)^{-\frac{d}{2}} \|M_0 f\|_{L^p(\R^d;Y)} \\ &\lesssim \bigl(\cos(\theta)\bigr)^{-\frac{d}{2}} \|M_{HL}f\|_{L^p(\R^d;Y)} \\ & \lesssim \bigl(\cos(\theta)\bigr)^{-\frac{d}{2}} \|f\|_{L^p(\R^d;Y)} . \end{align*}
\end{proof}

\begin{cor}
\label{cor-pure-laplacian}
Let $Y$ be any UMD lattice, $d \in \N$ and $1 < p < \infty$.
Then $A = - \Delta$ has a $\Hor^\beta_2$ functional calculus on $L^p(\R^d;Y)$ for any exponent $\beta > \frac{d+1}{2}$.
\end{cor}

\begin{remark}
The exponent $\beta$ is better than what gives Theorem \ref{thm-Hoermander} for a bad UMD lattice, i.e. with convexity $p_Y$ close to $1$ and concavity $q_Y$ close to $\infty$, or also if $Y$ has bad convexity and $p$ is close to $\infty$.
We do not know if for self-adjoint semigroups with (generalised) Gaussian estimates, the value of $\beta$ from Corollary \ref{cor-pure-laplacian} holds for general UMD lattices.
See Corollary \ref{cor-GE-Hoermander} for a partial answer in case of dispersive estimates.
\end{remark}

\begin{proof}[of Corollary \ref{cor-pure-laplacian}]
Applying Proposition \ref{prop-pure-laplacian} to $Y(\ell^2)$ which is a UMD lattice according to Lemma \ref{lem-UMD-lattice-convexity}, we deduce that for $\theta \in \bigl( - \frac{\pi}{2}, \frac{\pi}{2} \bigr)$, $t_1,t_2,\ldots,t_n > 0$ and $f_1,f_2,\ldots,f_n \in L^p(\R^d) \otimes Y \subseteq  L^p(\R^d;Y)$ that
\begin{align*}
\Biggl\| \biggl( \sum_i |T_{t_ie^{i\theta}} f_i|^2 \biggr)^{\frac12} \Biggr\|_{L^p(\R^d;Y)} & \leq \Biggl\| \biggl( \sum_i \bigl(M_\theta f_i\bigr)^2 \biggr)^{\frac12} \Biggr\|_{L^p(\R^d;Y)} \\
& \lesssim \bigl(\cos(\theta)\bigr)^{-\frac{d}{2}} \Biggl\| \biggl( \sum_i |f_i|^2 \biggr)^{\frac12} \Biggr\|_{L^p(\R^d;Y)},
\end{align*}
so that $\{T_{t e^{i\theta}} : \: t > 0 \}$ is $R$-bounded in $L^p(\R^d;Y)$ with $R$-bound $\lesssim \bigl(\cos(\theta)\bigr)^{-\frac{d}{2}}$.
Since $Y$ is UMD and $(T_t)_t$ is a positive contraction semigroup, $A$ has an $\HI(\Sigma_\omega)$ calculus on $L^p(\R^d;Y)$ for any $\omega > \frac{\pi}{2}$, according to \cite{Duong}.
Thus we can apply Theorem \ref{thm-KrW3} and deduce the H\"ormander functional calculus stated in the Corollary.
\end{proof}

With a little more technical effort, we can give a variant of the proof of Proposition \ref{prop-pure-laplacian} which works for self-adjoint semigroups with Gaussian estimates, with worse exponent $\alpha$.
We recall the following extrapolation result of Gaussian estimates from real to complex time \cite[Proposition 4.1]{CaCoOu}.

\begin{lemma}
\label{lem-CaCoOu}
Let $(\Omega,d,\mu)$ be a space of homogeneous type.
Let $p_t(x,y)$ be the integral kernel of a self-adjoint semigroup on $L^2(\Omega)$ with upper Gaussian estimate \eqref{equ-GE-prelims}, that is
\begin{equation}
\label{equ-GE}
 | p_t(x,y) | \leq C \frac{1}{V(x,t^{\frac1m})} \exp\biggl( -c \biggl( \frac{\dist(x,y)}{t^{\frac{1}{m}}} \biggr)^{\frac{m}{m-1}} \biggr) \quad (t > 0, x,y \in \Omega)
\end{equation}
with $m \geq 2.$
Then $p_t(x,y)$ has an analytic extension for $z = t \in \C_+$ and an estimate
\begin{multline*} |p_z(x,y)| \leq \\ \frac{C}{\biggl(V\biggl(x,\Bigl( \frac{|z|}{(\cos \theta)^{m-1}}\Bigr)^{\frac1m}\biggr)V\biggl(y,\Bigl( \frac{|z|}{(\cos \theta)^{m-1}}\Bigr)^{\frac1m}\biggr)\biggr)^{\frac12} }
\exp \biggl( - c \biggl( \frac{\dist(x,y)}{|z|^{\frac{1}{m}}} \biggr)^{\frac{m}{m-1}} \cos \theta \biggr) (\cos \theta)^{-d} \end{multline*}
where $z \in \C_+,\,x,y\in\Omega, \theta = \arg(z)$ and $d$ is a doubling dimension of $\Omega.$
\end{lemma}

\begin{prop}
\label{prop-Gaussian-maximal-operator}
Let $(\Omega,d,\mu)$ be a space of homogeneous type and $A$ a self-adjoint positive operator on $L^2(\Omega)$ with semigroup $T_t$ and kernel satisfying the upper Gaussian estimate \eqref{equ-GE}.
Let $Y = Y (\Omega')$ be a UMD lattice.
Let for $f \in L^p(\Omega) \otimes Y$ and $\theta \in \bigl(-\frac{\pi}{2},\frac{\pi}{2} \bigr)$ the maximal operator \[M_\theta ( f )(x,\omega') = \sup_{t > 0} \bigl| T_{te^{i\theta}} (f(\cdot,\omega'))(x)\bigr|.\]
Then the operator $M_\theta$ is bounded $L^p(\Omega;Y) \to L^p(\Omega;Y)$ for $1 < p < \infty$ with norm bound $\leq C (\cos \theta)^{-d}$, $d$ being a doubling dimension of $\Omega$.
\end{prop}

\begin{proof}
Let $\theta \in (-\frac{\pi}{2},\frac{\pi}{2}),\: 1 < p < \infty$ and $f \in L^p(\Omega) \otimes Y.$
We use Lemma \ref{lem-CaCoOu}.
Write $z = te^{i\theta}.$ Then

\begin{align*}
M_\theta f(x,\omega') & = \sup_{t > 0} \bigl| T_{te^{i\theta}} (f(\cdot,\omega'))(x)\bigr| \nonumber \\
& \leq \sup_{t > 0} \int_\Omega |p_{te^{i\theta}}(x,y)| |f(y,\omega')| \,dy \nonumber \\
& \leq \sup_{t > 0} \int_\Omega \frac{C}{\biggl(V\biggl(x,\Bigl( \frac{|z|}{(\cos \theta)^{m-1}}\Bigr)^{\frac1m}\biggr)V\biggl(y,\Bigl( \frac{|z|}{(\cos \theta)^{m-1}}\Bigr)^{\frac1m}\biggr)\biggr)^{\frac12} } \nonumber \\
&
\exp \biggl( - c \biggl( \frac{\dist^m(x,y)}{|z|} \biggr)^{\frac{1}{m-1}} \cos \theta \biggr) (\cos \theta)^{-d} |f(y,\omega')| \,dy, \nonumber \\
\end{align*}
and then 

\begin{multline} \label{equ-1-calculation-M_theta}
M_\theta f(x,\omega')  \leq\\   \sup_{t > 0} \int_\Omega \frac{C}{\Bigl(V\Bigl(x, t^{\frac1m}\Bigr)V\Bigl(y, t^{\frac1m}\Bigr)\Bigr)^{\frac12} }
\exp \biggl( - c \biggl( \frac{\dist^m(x,y)}{t} \biggr)^{\frac{1}{m-1}} \biggr) (\cos \theta)^{-d} |f(y,\omega')| \,dy. 
\end{multline}
Here we have simply performed the substitution $t \mapsto t (\cos \theta)^{m-1}.$
We decompose the integral over $\Omega$ in \eqref{equ-1-calculation-M_theta} into annular regions $A_n = B(x,2 t^{\frac1m})$ if $n = 0$ and $A_n = B(x,2^{n+1} t^{\frac1m}) \backslash B(x,2^n t^{\frac1m})$ if $n \geq 1.$
Then \eqref{equ-1-calculation-M_theta} continues
\begin{align}
& \leq C (\cos \theta)^{-d} \sup_{t > 0} \sum_{n = 0}^\infty \int_{A_n} \frac{1}{\Bigl(V\Bigl(x, t^{\frac1m}\Bigr)V\Bigl(y, t^{\frac1m}\Bigr)\Bigr)^{\frac12} }
\exp \Bigl( - c 2^{n\frac{m}{m-1}} \Bigr) |f(y,\omega')| \,dy \nonumber \\
& \leq C (\cos \theta)^{-d} \sup_{t > 0} \sum_{n = 0}^\infty \exp \Bigl( - c 2^{n\frac{m}{m-1}} \Bigr) \int_{B(x,2^{n+1} t^{\frac1m})} \frac{1}{\Bigl(V\Bigl(x, t^{\frac1m}\Bigr)V\Bigl(y, t^{\frac1m}\Bigr)\Bigr)^{\frac12} } |f(y,\omega')| \,dy.
\label{equ-2-calculation-M_theta}
\end{align}

Now we make use of Lemma \ref{lem-volume-homogeneous-type} to introduce in \eqref{equ-2-calculation-M_theta} the Hardy-Littlewood maximal operator.
Namely, we have for $y \in A_n$
\[
\frac{1}{\Bigl(V\Bigl(x, t^{\frac1m}\Bigr)V\Bigl(y, t^{\frac1m}\Bigr)\Bigr)^{\frac12}  }  = \frac{1}{V\Bigl(x, t^{\frac1m}\Bigr)} \sqrt{\frac{V\Bigl(x, t^{\frac1m}\Bigr)}{V\Bigl(y, t^{\frac1m}\Bigr)} } ,
\]
so that
\begin{align*}
\frac{1}{\Bigl(V\Bigl(x, t^{\frac1m}\Bigr)V\Bigl(y, t^{\frac1m}\Bigr)\Bigr)^{\frac12}  }  & \leq \frac{1}{V\Bigl(x, t^{\frac1m}\Bigr)} 2^{(n+1) d} \sqrt{\frac{V\Bigl(x, t^{\frac1m}\Bigr)}{V\Bigl(y,2^{(n+1)}t^{\frac1m}\Bigr)} } \\
& \leq \frac{1}{V\Bigl(x, t^{\frac1m}\Bigr)} 2^{(n+1) d} \sqrt{\frac{V\Bigl(x,2^{(n+1)}t^{\frac1m}\Bigr)}{V\Bigl(y,2^{(n+1)}t^{\frac1m}\Bigr)} } \\
& \leq C \frac{1}{V\Bigl(x, t^{\frac1m}\Bigr)} 2^{(n+1) d}.
\end{align*}

Thus, \eqref{equ-2-calculation-M_theta} continues as
\begin{align*}
& \lesssim (\cos \theta)^{-d} \sup_{t > 0} \sum_{n = 0}^\infty \exp \Bigl( - c 2^{n\frac{m}{m-1}} \Bigr) 2^{(n+1)d} \int_{B(x,2^{n+1} t^{\frac1m})} \frac{1}{V\Bigl(x, t^{\frac1m}\Bigr)} |f(y,\omega')| \,dy \\
&\lesssim (\cos \theta)^{-d} \sum_{n = 0}^\infty \exp \Bigl( - c 2^{n\frac{m}{m-1}} \Bigr) 2^{2(n+1)d} M_{HL}f(x,\omega') \\
& \lesssim (\cos \theta)^{-d} M_{HL} f(x,\omega'),
\end{align*}
where we have used \[\frac{1}{V\Bigl(x,t^{\frac1m}\Bigr)} \lesssim 2^{(n+1)d} \frac{1}{V\Bigl(x,2^{n+1}t^{\frac1m}\Bigr)}.\]
All these estimates were pointwise in $\omega' \in \Omega'.$
Since $Y = Y(\Omega')$ is a lattice, we thus have
\[ \|M_\theta f \|_{L^p(\Omega;Y)} \leq C (\cos \theta)^{-d} \bigl\|M_{HL}f  \bigr\|_{L^p(\Omega;Y)}. \]
Now the Proposition follows invoking Theorem \ref{thm-MHL}.
\end{proof}

As in the pure Laplacian case, there is the following Corollary on H\"ormander functional calculus.
Note however that this time, we have to assume a priori that $A$ has an $\HI$ calculus on $L^p(\Omega;Y)$.
For the existence of the $\HI$ calculus, we refer to Theorems \ref{thm-HI-extrapolation} and \ref{thm-HI-Kunstmann-Ullmann}, and Proposition \ref{prop-HI-diffusion}.
Moreover, part 1. is entirely covered by Theorem \ref{thm-Hoermander}.

\begin{cor}
\label{cor-GE-Hoermander}
Let $(\Omega,\dist,\mu)$ be a space of homogeneous type of dimension $d$, let $Y$ be any UMD lattice and let $1 < p < \infty$.
Assume that $A$ generates the self-adjoint semigroup $(T_t)_t$ on $L^2(\Omega)$ and that $A$ has a bounded $\HI(\Sigma_\omega)$ calculus on $L^p(\Omega;Y)$ for some $\omega \in (0,\pi)$.
\begin{enumerate}
\item
Assume that $(T_t)_t$ has Gaussian estimates \eqref{equ-GE-prelims}.
Then the $\HI(\Sigma_\omega)$ calculus improves to a H\"ormander $\Hor^\beta_2$ functional calculus on $L^p(\Omega;Y)$ for any exponent $\beta > d + \frac{1}{2}$.
\item
Assume that $(T_t)_t$ satisfies Davies-Gaffney estimates, that is, generalised Gaussian estimates \eqref{equ-GGE} with $p_0 = 2$, $m = 2$.
Assume moreover that $(T_t)_t$ has an integral kernel $p_t$ satisfying the on-diagonal estimate
\[ p_t(x,x) \leq C t^{-\frac{d}{2}} \]
for some $C > 0$ and all $x \in \Omega$ and $t > 0$, and
that there exists $C < \infty$ such that 
\[ \|T_t\|_{L^\infty(\Omega) \to L^\infty(\Omega)} \leq C\]
for all $t \geq 0$.
Assume finally that the volume satisfies a polynomial growth 
\[ V(x,r) \leq C |r|^d \]
for some fixed $C > 0$ and any $x \in \Omega$ and $r > 0$.

All the assumptions in 2. up to now are satisfied e.g. if $\Omega$ is of polynomial volume growth of dimension $d$, and $T_t$ satisfies Gaussian estimates \eqref{equ-GE-prelims} with $m = 2$.
If the semigroup satisfies the dispersive estimate
\[ \|\exp(itA)\|_{L^1(\Omega) \to L^\infty(\Omega)} \leq C |t|^{-\frac{d}{2}}\]
for some $C > 0$ and all $t \in \R \backslash \{ 0 \}$,
then the $\HI(\Sigma_\omega)$ calculus improves to a H\"ormander $\Hor^\beta_2$ functional calculus on $L^p(\Omega;Y)$ for any exponent $\beta > \frac{d}{2} + \frac12$.
\end{enumerate}
\end{cor}

\begin{proof}
1. The proof goes along the same lines as Corollary \ref{cor-pure-laplacian}, using Proposition \ref{prop-Gaussian-maximal-operator} in place of Proposition \ref{prop-pure-laplacian}, and the a priori existence of the $\HI$ calculus.

2. Observe that rescaling the time $t \leadsto ct$ in the semigroup if necessary, we have a semigroup satisfying \cite[(3.1), (3.2)]{CouSi}, where we use \cite[Lemma 3.2]{CouSi}.
According to \cite[p.~521-522]{CouSi}, the dispersive assumption and the on-diagonal kernel estimate imply
\[ |p_z(x,y)| \leq C |z|^{-\frac{d}{2}} \exp\biggl(- \Re\biggl[ \frac{\dist^2(x,y)}{4z} \biggr] \biggr) \leq C' \frac{1}{\sqrt{V(x,|z|)}} \exp\biggl(- \Re\biggl[ \frac{\dist^2(x,y)}{4z} \biggr] \biggr).\]
Now we can argue, first as in the proof of Proposition \ref{prop-pure-laplacian} and second as in the proof of Lemma \ref{prop-Gaussian-maximal-operator} to see that $\| M_\theta \|_{L^p(\Omega;Y) \to L^p(\Omega;Y)} \leq C \bigl(\cos(\theta)\bigr)^{-\frac{d}{2}} \|M_{HL}\|_{L^p(\Omega;Y) \to L^p(\Omega;Y)}$.
Then use the $\HI$ calculus assumption to deduce as in Corollary \ref{cor-pure-laplacian} that $A$ has a $\Hor^\beta_2$ calculus on $L^p(\Omega;Y)$ for $\beta > \frac{d+1}{2}$.
\end{proof}

\begin{remark}
\label{rem-interpolation-dispersive}
Let the assumptions of Corollary \ref{cor-GE-Hoermander} 2. hold.
There holds a similar Remark to \ref{rem-interpolation}.
Namely, assume that $Y(\Omega') = [Z(\Omega'),L^2(\Omega')]_\theta$ with $Z$ a further UMD lattice and $\theta \in (0,1)$.
Interpolating between the calculus mappings
\begin{align*}
\Hor^{\frac{d+1}{2} + \epsilon}_2  & \to B( L^{p_1}(\Omega;Z)), \: f \mapsto f(A) \\
\intertext{and}
\Hor^{\frac12 + \epsilon}_2 & \to B( L^2(\Omega;L^2(\Omega'))) , \: f \mapsto f(A)
\end{align*}
with $p_1$ close to $1$ resp. $\infty$,
one gets for $\frac{1}{p} \geq 1 - \frac{\theta}{2} ( = \frac{1}{p_Y}$ typically) resp. $\frac{1}{p} \leq \frac{\theta}{2} ( = \frac{1}{q_Y}$ typically) that $A$ admits a $\Hor^\beta_2$ calculus on $L^p(\Omega;Y)$ with $\beta > d | \frac1p - \frac12 | + \frac12$.
Again as in Remark \ref{rem-interpolation}, one can slightly improve the index $\beta$ by considering the classes $\Hor^\beta_q$ and taking $\Hor^{\epsilon}_\infty \to B( L^2(\Omega;L^2(\Omega')))$ in place of $\Hor^{\frac12 + \epsilon}_2 \to B( L^2(\Omega;L^2(\Omega')))$ above.
Taking into account all possible values for $p,p_Y,q_Y$, we get the following:
If $Y$ is a UMD lattice with convexity resp. concavity exponents $p_Y \in (1,2]$ resp. $q_Y \in [2,\infty)$ such that we can write moreover $Y(\Omega') = [Z(\Omega'),L^2(\Omega')]_\theta$ with $Z(\Omega')$ a UMD lattice and $\theta = 2 \min\bigl(1-\frac{1}{p_Y},\frac{1}{q_Y}\bigr)$, then
$A$ has a $\Hor^\beta_2$ calculus on $L^p(\Omega;Y)$ for $\beta > \widetilde{\alpha} \cdot d + \frac12$ with
\[ \widetilde{\alpha} = \widetilde{\alpha}(p,p_Y,q_Y) = \max\biggl(\Bigl|\frac1p - \frac12\Bigr|, \Bigl|\frac1{p_Y}-\frac12\Bigr|,\Bigl|\frac1{q_Y} - \frac12\Bigr|\biggr) \leq \alpha(p,p_Y,q_Y) . \]
\end{remark}

\section{Examples}
\label{sec-examples}

The next lemma gives a simple procedure to create the UMD lattices that are needed in our Main Theorem \ref{thm-Hoermander} out of given ones.

\begin{lemma}
Let $1 < p < \infty$.
\begin{enumerate}
\item Suppose that $Y_1,Y_2$ are $p$-convex UMD lattices.
Then also $Y_1(Y_2)$ is a $p$-convex UMD lattice.
In particular, $L^q(\Omega;Y_1)$ is a $p$-convex UMD lattice for any $q \geq p$.
\item Suppose that $Y_1,Y_2$ are $p$-concave UMD lattices.
Then also $Y_1(Y_2)$ is a $p$-concave UMD lattice.
In particular, $L^q(\Omega;Y_1)$ is a $p$-concave UMD lattice for any $q \leq p$.
\item If $1 < q < p$ and $Y$ is a $q$-convex Banach lattice, then $[L^p(\Omega;Y)]^q = L^{\frac{p}{q}}(\Omega;Y^q)$.
In particular, if $Y^q$ is UMD, then also $[L^p(\Omega;Y)]^q$ is UMD.
\end{enumerate}
\end{lemma}

\begin{proof}
1.  We have
\begin{align*}
\Biggl\| \biggl( \sum_i |f_i|^p \biggr)^{\frac1p} \Biggr\|_{Y_1(Y_2)} & = \Biggl\| \ \Biggl\| \omega' \mapsto \biggl( \sum_i |f_i(\omega')|^p \biggr)^{\frac1p} \Biggr\|_{Y_2}\  \Biggr\|_{Y_1} \\
& \leq C \Biggl\| \omega' \mapsto \biggl( \sum_i \|f_i(\omega')\|_{Y_2}^p \biggr)^{\frac1p} \Biggr\|_{Y_1} \\
& \leq C C' \biggl( \sum_i \|f_i\|_{Y_1(Y_2)}^p \biggr)^{\frac1p}.
\end{align*}

2. We use the part 1. and the fact from Lemma \ref{lem-UMD-lattice-convexity} that $Y_1'$ and $Y_2'$ are $p'$-convex.

3. We have
\begin{align*}
\|f\|_{[L^p(\Omega;Y)]^q} & = \Bigl\| \: |f|^{\frac1q} \Bigr\|_{L^p(\Omega;Y)}^q \\
& = \biggl( \int_\Omega \bigl\|\: |f(x)|^{\frac1q} \bigr\|_Y^p \,dx \biggr)^{\frac{q}{p}} \\
& = \biggl( \int_\Omega \|f(x)\|_{Y^q}^{\frac{p}{q}} \,dx \biggr)^{\frac{q}{p}} \\
& = \|f\|_{L^{\frac{p}{q}}(\Omega;Y^q)}.
\end{align*}
\end{proof}

\subsection{Gaussian estimates}

In this subsection, we show that for many examples of differential operators in different contexts, Gaussian estimates and $\HI$ calculus on $L^p(\Omega;Y)$ are available, and thus Theorem \ref{thm-Hoermander} on the H\"ormander calculus on $L^p(\Omega;Y)$ applies for $1 < p < \infty$ and $Y$ a UMD lattice.
We recall that the derivation exponent $\alpha = \alpha(p,p_Y,q_Y) \in (0,1)$ is given in \eqref{equ-defi-alpha} and $p_Y$ resp. $q_Y$ define the convexity and concavity exponents in $(1,\infty)$ of the lattice $Y$ to which $Y^{p_Y}$ and $(Y')^{q_Y'}$ are still UMD lattices.
For example, if $Y = L^s(\Omega')$ for some $s \in (1,\infty)$, then any $p_Y \in (1,s)$ and any $q_Y \in (s,\infty)$ are admissible.

\paragraph{Manifolds}

Let $\Omega = M$ be a complete Riemannian manifold with non-negative Ricci curvature.
Then the heat semigroup (associated with the Laplace-Beltrami operator) is a symmetric contraction semigroup with Gaussian estimates \eqref{equ-GE-prelims} of order $m = 2$.
See \cite{LY}, \cite[p. 3/70 (1.3)]{GriTel}, \cite{Sal}.
Hence on these manifolds, according to Corollary \ref{cor-Hoermander} and Proposition \ref{prop-HI-diffusion}, the heat semigroup has a H\"ormander $\Hor^\beta_2$ calculus on $L^p(\Omega;Y)$ for $1 < p < \infty$, for any UMD lattice $Y$ and $\beta > \alpha(p,p_Y,q_Y) \cdot d + \frac12$.

 \paragraph{Schr\"odinger and differential operators}

We show now that our main results apply for several Schr\"odinger operators.

Start with the case $\Omega = M$ is a connected and complete Riemannian manifold with non-negative Ricci curvature.
Consider a potential $V : \Omega \to \R$ such that $V \geq 0$ and $V \in L^1_{loc}(\Omega)$.
Then $A = - \Delta + V$, defined by the quadratic form technique, generates a self-adjoint semigroup $(T_t)_t$ on $L^2(\Omega)$, and moreover, as a consequence of the Trotter-Kato product formula, $|T_tf(x)| \leq S_t|f|(x)$, where $S_t = \exp(t\Delta)$ is the heat semigroup \cite[Section 7.4]{DuOS}.
According to the preceding paragraph on manifolds, $S_t$ is $L^1$ and $L^\infty$ contractive, so according to Proposition \ref{prop-HI-diffusion}, $A$ has an $\HI$ calculus on $L^p(\Omega;Y)$ for $1 < p <\infty$ and $Y$ any UMD lattice.
Moreover, $(T_t)_t$ has Gaussian estimates \eqref{equ-GE-prelims} of order $m = 2$ \cite[(7.8)]{DuOS}, so that according to Corollary \ref{cor-Hoermander}, $A$ has a $\Hor^\beta_2$ calculus on $L^p(\Omega;Y)$ with $\beta > \alpha(p,p_Y,q_Y) \cdot d + \frac12$.

Now consider the case that $\Omega \subset \R^d$ is an open subset of homogeneous type.
Take the following self-adjoint differential operator defined on $L^2(\Omega)$ \cite[(1)]{Ouh06}:
\[ A = - \sum_{k,j = 1}^d\frac{\partial}{\partial x_j} \Bigl( a_{kj} \frac{\partial}{\partial x_k} \Bigr) \]
where $a_{kj} = a_{jk} \in L^\infty(\Omega,\R), \: 1 \leq k,j \leq d$ and $a_{kj}$ satisfy the standard ellipticity condition $\eta I \leq (a_{kj})_{kj} \leq \mu I$ for some constants $0 < \eta < \mu < \infty$.
We assume Dirichlet boundary conditions.
Then according to \cite[Theorem 1]{Ouh06}, the semigroup $(T_t)_t$ generated by $A$ is positive and according to \cite[(4)]{Ouh06}, satisfies Gaussian estimates \ref{equ-GE-prelims} with $m = 2$ (note that $V(x,\sqrt{t}) \leq C t^{\frac{d}{2}}$ there).
Thus according to Corollary \ref{cor-HI-positive}, $A$ has an $\HI$ calculus on $L^p(\Omega;Y)$ for $Y$ any UMD lattice, and according to Corollary \ref{cor-Hoermander}, $A$ has a $\Hor^\beta_2$ calculus on $L^p(\Omega;Y)$ for $\beta > \alpha(p,p_Y,q_Y) \cdot d + \frac12$.

Now consider the case that $\Omega = \R^d$ and a potential $V : \Omega \to \R$ such that $V = V^+ - V^-$, $V^+,V^- \geq 0$ and $V^+,V^-$ belong to the Kato class (see \cite{Sim}).
Then according to \cite[Corollary 3]{Ouh06}, $A = - \Delta + V$, the self-adjoint Schr\"odinger operator with potential $V$, generates a positive semigroup $(T_t)_{t}$ (even with a certain lower Gaussian estimate), and according to \cite[Theorem 1]{Ouh06}, the shifted semigroup generated by $A - s(A) + \epsilon$ for an $\epsilon > 0$ has Gaussian estimates \eqref{equ-GE-prelims}.
Here, $s(A) = \inf \sigma(A)$ is the spectral bound of $A$ in $L^2(\R^d)$.
Thus, according to Corollaries \ref{cor-HI-positive} and \ref{cor-Hoermander}, $A - s(A) + \epsilon$ has a $\Hor^\beta_2$ calculus on $L^p(\Omega;Y)$ for $\beta > \alpha(p,p_Y,q_Y) \cdot d + \frac12$.

Now consider for $\lambda > 0$ the Bessel operator \[A = \Delta_\lambda = - x^{-\lambda} \frac{d}{dx} x^{2 \lambda} \frac{d}{dx} x^{-\lambda} =  - \frac{d^2}{dx^2} + \lambda ( \lambda - 1 ) x^{-2}\] on $\Omega = (0,\infty)$ \cite[p.~343]{BCRM}.
Then according to \cite{Ouh95}, the semigroup $(T_t)_t$ generated by $A$ satisfies Gaussian estimates \eqref{equ-GE-prelims} provided that the potential $\lambda (\lambda - 1) x^{-2}$ is positive, i.e. $\lambda \geq 1$.
According to \cite[Theorem 1.6]{BCRM}, $A$ has an $\HI$ calculus on $L^p(\Omega;Y)$ for $1 < p < \infty$ and any UMD space $Y$ that is a complex interpolation space between a Hilbert space and another UMD space.
According to Lemma \ref{lem-Tomczak}, any UMD lattice is of this form.
Thus, for the particular case $Y$ being a UMD lattice and $\lambda \geq 1$, we can strengthen \cite[Theorem 1.6]{BCRM} and deduce for $\Delta_\lambda$ a $\Hor^\beta_2$ calculus on $L^p(\Omega;Y)$ for $\beta > \alpha(p,p_Y,q_Y) \cdot d + \frac12$.

There are other Schr\"odinger and differential operators, where Gaussian estimates are available and the semigroup is positive, hence Corollaries \ref{cor-HI-positive} and \ref{cor-Hoermander} apply.
We refer to \cite{Ouh06}, \cite[Section 6.4, in particular Theorems 6.10, 6.11]{Ouh} for upper Gaussian estimates, and for lower Gaussian estimates \cite[Section 7.8]{Ouh06}.

\paragraph{Lie groups of polynomial volume growth}

Consider $\Omega = G$ a Lie group having polynomial volume growth.
Then $\Omega$ is a space of homogeneous type.
Consider moreover $A = - \sum_{k = 1}^N X_k^2$, where $\{X_1, \ldots , X_N\}$ is a family of left invariant vector fields having the H\"ormander property.
For example, $G = \R^{2n + 1}$ is the Heisenberg group, and $A = - \sum_{k = 1}^n X_k^2 + Y_k^2$ is the standard sub-Laplacian.
Then according to \cite[Theorem 4.2, Example 2]{Sal}, \cite{Gri}, the semigroup $(T_t)_t$ satisfies two-sided Gaussian estimates \eqref{equ-GE-prelims} with $m = 2$.
Therefore, according to Corollories \ref{cor-HI-positive} and \ref{cor-Hoermander}, $A$ has a $\Hor^\beta_2$ calculus on $L^p(\Omega;Y)$ for $\beta > \alpha(p,p_Y,q_Y) \cdot d + \frac12$.

\paragraph{Fractals}

There are several fractals $\Omega \subseteq \R^n$ on which there exists a heat semigroup satisfying upper and lower Gaussian estimates.
Namely, one first turns $\Omega$ into a metric measure space by choosing a metric, e.g. the intrinsic metric inherited from $\R^n$ and a Hausdorff measure.
Then, the heat generator $A$ is defined using the form method, often by means of a Brownian motion Dirichlet form \cite[preprint version p.~3-4]{GriTel}.
The heat kernel $p_t(x,y)$ satisfies \cite[(1.4)]{GriTel}
\[ p_t(x,y) \cong \frac{C}{t^{\alpha/\beta}} \exp\Biggl(-c \biggl( \frac{d^\beta(x,y)}{t} \biggl)^{\frac{1}{\beta - 1}} \Biggr) \]
for certain $\alpha > 0$ and $\beta > 1$ ($\beta \geq 2$ according to \cite[Abstract]{GHL})), and the implied constant $c$ may be different between upper and lower estimate.
The parameter $\beta$ is called walk dimension.
In case of volume comparability $V(x,t) \cong t^\alpha$, e.g. if $A$ is the Laplace operator on the Sierpinski Gasket \cite[Section 7.11]{DuOS}, we can apply our Corollaries \ref{cor-HI-positive} and \ref{cor-Hoermander} to deduce that $A$ has a $\Hor^\beta_2$ calculus on $L^p(\Omega;Y)$ for $1 < p < \infty$ and $Y$ a UMD lattice, with $\beta > \alpha(p,p_Y,q_Y) \cdot d + \frac12$.

For a discussion of many further examples where Gaussian estimates as in \eqref{equ-GE-prelims} are satisfied, we refer to \cite[Section 7]{DuOS}.
Hence in all these cases, Theorem \ref{thm-HI-Kunstmann-Ullmann} and Corollary \ref{cor-Hoermander} are applicable and we obtain for the operators $A$ a bounded $\Hor^\beta_2$ calculus on $L^p(\Omega;L^s(\Omega'))$ with $1 < p,s < \infty$ and $\beta > \bigl( \max(\frac1p , \frac1s , \frac12) - \min (\frac1p, \frac1s, \frac12) \bigr) \cdot d + \frac12$.

\subsection{Dispersive estimates}

In this subsection, we indicate in which situations of the preceding subsection there is a dispersive estimate
\begin{equation}
\label{equ-dispersive}
\| \exp(itA) \|_{L^1(\Omega) \to L^\infty(\Omega)} \leq C |t|^{-\frac{d}{2}} \quad (t \in \R)
\end{equation}
available, so that Corollary \ref{cor-GE-Hoermander} 2. is applicable, and we can deduce a $\Hor^\beta_2$ calculus on $L^p(\Omega;Y)$ for $\beta > \frac{d}{2} + \frac12$.
This is a smaller differentiation order, hence a better result than what we had obtained in the preceding subsection, in case that $\alpha(p,p_Y,q_Y) > \frac12$, e.g. if $p$ is close to $\infty$ and $Y$ is an $L^s(\Omega')$ space with $s$ close to $1$.

\paragraph{Schr\"odinger operators}

Throughout the paragraph, we assume $\Omega = \R^d$ and $A = - \Delta + V$ a Schr\"odinger operator with positive locally integrable potential.

First, consider the case $d = 1$.
Then if $\int_\R V(x) (1 + |x|) \,dx < \infty$, if there is no resonance at zero energy and if there are no bound states (which implies that the spectral projection $P_{ac}(A)$ onto the absolutely continuous spectral subspace is the identity), then according to 
\cite[Theorem 1]{GoS}, $A$ satisfies \eqref{equ-dispersive}.
Consequently, $A$ has a $\Hor^\beta_2$ calculus on $L^p(\R;Y)$ for any $\beta > \frac12 + \frac12 = 1$.

Second, consider the case $d = 3$.
Then if $V(x) \leq C ( 1 + |x| )^{-b}$ for some $b > 3$ and all $x \in \R^3$, if $0$ is neither an eigenvalue of $A$ nor a resonance, and if there are no bound states, then according to \cite[Theorem 2]{GoS}, \eqref{equ-dispersive} holds.
Consequently, $A$ has a $\Hor^\beta_2$ calculus on $L^p(\R^3;Y)$ for any $\beta > \frac32 + \frac12 = 2$.

Next, consider $d \in \N$ an arbitrary odd value.
Then if $V \in C^{\frac{d-3}{2}}(\R^d)$ for $d \in \{ 5 , 7 \}$, if $V(x) \leq c (1 + |x|)^{-b}$ for some $b > \frac{3d + 5}{2}$ and for $1 \leq j \leq \frac{d-3}{2}$, $|\nabla^j V(x)| \leq c (1 + |x|)^{-a}$ for some $a > 3$ for $d = 5$ and for some $a > 8$ for $d = 7$, if $0$ is not an eigenvalue of $A$ and if there are no bound states, then according to \cite[Theorem 1.1]{ErGr}, \eqref{equ-dispersive} holds.
Consequently, $A$ has a $\Hor^\beta_2$ calculus on $L^p(\R^d;Y)$ for any $\beta > \frac{d}{2} + \frac12$.

Now, if $V$ is of the form $V(x_1,\ldots,x_d) = W(x_1) + W(x_2) + \ldots + W(x_d)$ with $W : \R \to \R_+$ such that $\int_\R W(x) (1 + |x|)^2 \,dx < \infty$, then according to \cite[Corollary 1.6]{Pier}, \eqref{equ-dispersive} holds.
Consequently, $A$ has a $\Hor^\beta_2$ calculus on $L^p(\R^d;Y)$ for any $\beta > \frac{d}{2} + \frac12$.

\paragraph{Stratified Lie groups}

We refer to the recent work \cite{BFG} for a study when \eqref{equ-dispersive} or a stronger estimate holds in the case that $\Omega = G$ is a  2-step stratified Lie group with further properties and $A = - \Delta$ is the Laplace-Beltrami operator.

\subsection{Generalised Gaussian estimates}
\label{subsec-GGE}

In the recent past, several operators with generalised Gaussian estimates \eqref{equ-GGE} for some $p_0 > 1$ have been studied.
In these cases we will obtain according to Theorems \ref{thm-HI-Kunstmann-Ullmann} and \ref{thm-Hoermander} that $A$ has a $\Hor^\beta_2$ calculus on $L^p(\Omega;L^s(\Omega'))$ for $p_0 < p,s < p_0'$ and 
\begin{equation}
\label{equ-beta-GGE}
\beta > \Biggl( \max\biggl(\frac1p , \frac1s , \frac12\biggr) - \min \biggl(\frac1p, \frac1s, \frac12\biggr) \Biggr) \cdot d + \frac12 .
\end{equation}

\paragraph{Elliptic operators in divergence form}

Suppose that $\Omega = \R^d$ and $A$ is given by
\[ A f = \sum_{|\gamma|,|\delta|= m } (-1)^{|\delta|} \partial^\delta(a_{\gamma \delta}\partial^\gamma f) , \]
where $a_{\gamma \delta} \in L^\infty(\Omega;\R)$.
We suppose that the form $\mathfrak{a}$ associated with $A$, given by
\[ \mathfrak{a}(f,g) = \int \sum_{|\gamma|,|\delta| = m } a_{\gamma \delta} (x) \partial^{\gamma}f(x) \overline{\partial^\delta g(x)} \,dx \]
gives rise to a self-adjoint operator and satisfies the ellipticity condition
\[ \mathfrak{a}(f,f) \geq \eta \bigl\| ( - \Delta )^{\frac{m}{2}} f \bigr\|_2^2 \quad (f \in W^{m,2}(\R^d) ) \]
for some $\eta > 0$.
Then according to \cite[Section 3 a) (iii)]{KuUl}, \eqref{equ-GGE} holds with $m$ replaced by $2m$ and $p_0 = p_1'$, where $p_1 = \frac{2d}{d - 2m}$ for $d > 2m$, and $p_1 = \infty$ if $d < 2m$.
Consequently, $A$ has a $\Hor^\beta_2$ calculus on $L^p(\R^d;L^s(\Omega'))$ with $p_0 < p,s < p_0'$ and $\beta$ given by \eqref{equ-beta-GGE}.

\paragraph{Schr\"odinger operators with singular potentials}

Suppose again that $\Omega = \R^d$, with $d \geq 3$, and that $A = - \Delta + V$ is a Schr\"odinger operator.
We suppose that $V = V^+ - V^-$ with $V^+,V^- : \R^d \to \R_+$ and that $V^+$ is locally integrable and $V^-$ belongs to the pseudo-Kato class \cite{KPS}.
A typical example is $V(x)  = - \frac{c}{|x|^2}$ for a certain range of $c > 0$ \cite{KPS,KuUl}.
Then $A$ is self-adjoint, and according to \cite[Section 3 (c) (ii)]{KuUl}, \eqref{equ-GGE} holds for some $p_0 > 1$.
Consequently, $A$ has a $\Hor^\beta_2$ calculus on $L^p(\R^d;L^s(\Omega'))$ for any $p_0 < p,s < p_0'$ and $\beta$ as in \eqref{equ-beta-GGE}.

We refer to \cite[Section 2]{Bl},\cite[Section 3]{KuUl} and the references therein for detailed explanations of the two preceding paragraphs and more examples.

\section{Concluding remarks}

In Theorems \ref{thm-HI-Kunstmann-Ullmann} and \ref{thm-HI-extrapolation}, Proposition \ref{prop-HI-diffusion} and Corollary \ref{cor-HI-positive}, we gave some sufficient conditions, when $A$ has an $\HI$ calculus on $L^p(\Omega;Y)$.
Nevertheless, it would be interesting to know whether generalised Gaussian estimates and self-adjointness of the semigroup $T_t$ imply already themselves that $A$ has an $\HI$ calculus (and thus a $\Hor^\beta_2$ calculus) on $L^p(\Omega;Y)$ provided that $Y$ is a $p_0$-convex and $p_0'$-concave UMD lattice.
Already the case of classical Gaussian estimates and self-adjointness is open here (then the convexity and concavity assumption on $Y$ is void).

Another question is whether Theorems \ref{thm-Hoermander-intro}, \ref{thm-Hoermander-square-intro} and \ref{thm-dispersive-intro} hold for $Y$ being an intermediate UMD space, that is, $Y = [L^2(\Omega'),Z]_\theta$ for some further UMD space $Z$ and $\theta \in (0,1)$, or even for $Y$ being any UMD space.
Then, we suspect that the convexity and concavity notions, which only make sense for lattices, have to be replaced by Rademacher type and cotype.
For an $\HI(\Sigma_\omega)$ functional calculus result on $L^p(\Omega;Y)$ with $Y$ an intermediate UMD space and an estimate for the angle $\omega < \frac{\pi}{2}$, we refer to \cite[Theorem 1.6]{BCRM} with a particular Bessel operator $A$, and \cite[Theorem 4]{Xu15} for regular contractive and analytic semigroups.

A further question is whether a version of Theorem \ref{thm-HI-extrapolation} holds for generalised instead of classical Gaussian estimates.

We finally remark that the question about optimal exponents $\beta,q$ in $\Hor^\beta_q$ calculus with
\[ \|f \|_{\Hor^\beta_q} = |f(0)| + \sup_{R > 0} \|\phi f(R\cdot)\|_{W^\beta_q(\R)} \]
similar to Definition \ref{defi-Hoermander-class}
is even open in the scalar case $L^p(\Omega)$ within the class of all self-adjoint semigroups with Gaussian estimates.
Moreover, for $|\frac1p - \frac12| < \frac{1}{d+1}$ and $q = 2$, the best $\beta$ seems to be unknown even in the pure Laplacian case on $\Omega = \R^d$ see \cite[Subsection 8.1]{KrW3}.

\section{Acknowledgments}

The first and third author are financially supported by the grant ANR-18-CE40-0021 of the French National Research Agency ANR (project HASCON). The first author is financially supported by the grant ANR-18-CE40-0035  (project REPKA) and
the third author is financially supported by the grant ANR-17-CE40-0021 (project Front).

\vspace{0.2cm}

\footnotesize{
\noindent Luc Deleaval \\
\noindent
Laboratoire d'Analyse et de  Math\'ematiques Appliqu\'ees (CNRS UMR 8050)\\
Universit\'e Paris-Est Marne la Vall\'ee \\
5, Boulevard Descartes, Champs sur Marne\\
77454 Marne la Vall\'ee Cedex 2\\
luc.deleaval@u-pem.fr\hskip.3cm
}

\vspace{0.2cm}

\footnotesize{
\noindent Mikko Kemppainen \\
\noindent
mikko.kullervo.kemppainen@gmail.com\hskip.3cm
}

\vspace{0.2cm}
\footnotesize{
\noindent Christoph Kriegler (corresponding author)\\
\noindent
Laboratoire de Math\'ematiques Blaise Pascal (CNRS UMR 6620)\\
Universit\'e Clermont Auvergne\\
63 000 Clermont-Ferrand, France\\
christoph.kriegler@math.univ-bpclermont.fr\hskip.3cm
}
\end{document}